\documentclass[12pt]{amsart}
\usepackage{amssymb,amsthm}
\usepackage[T1]{fontenc}
\usepackage[bitstream-charter]{mathdesign}
\usepackage{graphicx}
\usepackage{tikz}
\usepackage{mathabx}
\usepackage{hyperref}

\numberwithin{equation}{section}

\newtheorem{thm}{Theorem}[section]
\newtheorem{lem}[thm]{Lemma}
\newtheorem{prop}[thm]{Proposition}

\newtheorem{cor}[thm]{Corollary}

\newtheorem{exam}[thm]{Example}
\newtheorem{remark}[thm]{Remark}

\renewcommand{\Re}{\mathrm{Re}}

\newcommand{\FN}{\mathbb{N}}
\newcommand{\FZ}{\mathbb{Z}}  % Integer ring Z
\newcommand{\FR}{\mathbb{R}}  % Real field
\newcommand{\FQ}{\mathbb{Q}}  % Rational field

\newcommand{\fa}{\mathfrak{a}}
\renewcommand{\a}{\mathfrak{a}}
\newcommand{\fb}{\mathfrak{b}}
\newcommand{\fc}{\mathfrak{c}}
\newcommand{\ff}{\mathfrak{f}}

\newcommand{\fp}{\mathfrak{p}}

\newcommand{\Hom}{\operatorname{Hom}}
\newcommand{\td}{\operatorname{Todd}}
\newcommand{\Res}{\operatorname{Res}}

\title[Special values of zeta functions]{Special values of partial zeta functions of real quadratic fields
at nonpositive integers and Euler-Maclaurin formula}
\author{Byungheup Jun and Jungyun Lee}
\email{byungheup@gmail.com} \email{lee9311@kias.re.kr}
\address{School of Mathematics,  Korea  Institute for Advanced Study\\
Hoegiro 87, Dongdaemun-gu, Seoul 130-722, Korea}

\begin{document}

\begin{abstract}
We compute the special  values at nonpositive integers of 
%special values of ideal class 
the partial zeta function of an ideal of a real quadratic field applying an asymptotic 
version of Euler-Maclaurin formula to the lattice cone associated to the ideal considered.
The Euler-Maclaurin formula involved is 
obtained by applying  the Todd series of differential operators to an integral of a small perturbation of the
cone. 
%a perturbed integral followed by application of Todd  series of differential operators arising from the cone. 
The additive property of Todd series w.r.t. the cone decomposition enables us to express
the partial zeta values in terms of the continued fraction of the reduced element of the ideal. 
The expression obtained uses the positive continued fraction which yields a virtual decomposition of the cone. 

We apply the expression to some indexed families of real quadratic fields satisfying certain condition
on the shape of the continued fractions.
The families considered include those appeared in \cite{J-L1} and \cite{J-L2} as well as the Richaud-Degert types. We show that the partial zeta values at a given nonpositive integer $-k$  in the family
indexed by $n$ is a polynomial of $n$. 

Finally, we compute explicitly the polynomials producing the partial zeta values at $s=-k$  for small $k$ of  some chosen families and compare these with some previously known results.
%We find a new expression of partial zeta function of real quadratic field using positive continued fraction. We make a virtual cone decomposition of two dimensional cone using positive continued fraction and combine it with additivity of todd power series.

\end{abstract}

%\date{2011.6.27}
%\date{2012.5.31}
%\date{2012.7.2}
%\date{2012.8.7}
%\date{2012.8.31}
\date{2012.9.18}
\maketitle
\tableofcontents

\section{Introduction}
Let $K$ be a number field of the extension degree $[K:\FQ]=r_1+2 r_2$, where $r_1$ and $r_2$
denote respectively the number of real and complex embeddings of $K$.
The Dedekind zeta function 
$$
\zeta_K(s) = \prod_{\fp: \text{prime ideal in $K$}} \frac{1}{1- N\fp^{-s}}
$$
is encoded with many interesting arithmetic properties of $K$. In particular, the residue at $s=1$ is 
associated to the class number $h_K$ of $K$ by the class number formula:
$$
\Res_{s=1} \zeta_K(s) = \frac{2^{r_1} (2\pi)^{r_2} R_K h_K}{\omega_K \sqrt{|D_K|}},
$$
where $R_K$ is the regulator, $\omega_K$ is the number of roots of $1$ in $K$ and 
$D_K$ is the discriminant.
This has been the starting point of most studies of class numbers. 

The simplest  is the case of imaginary quadratic fields where the regulator appears to be trivial. In \cite{Gauss},
Gauss listed 9 imaginary quadratic fields of class number 1 and conjectured that the list is complete. 
Later on this had been studied through 20th century and is now quite well understood
and solved by works of Heegner, Stark, Goldfeld and several others(eg. \cite{Baker}, \cite{G-Z}, \cite{Stark}, \cite{Goldfeld},
\cite{Goldfeld2}, \cite{Heegner}, \cite{Stark} and \cite{Stark2}). 

The case of real quadratic fields is more complicated  due to the presence of nontrivial regulator.
It is also conjectured by Gauss that there are infinitely many real quadratic fields of class number one. 
But since the regulator is far from being controlled in relation to the discriminant, and there has been no essential 
progress to the proof of the conjecture.

Instead of treating the whole real quadratic fields, people considered some families of
real quadratic fields where the regulators are controlled in relation to the discriminant.
The most well-known family of this kind is the Richaud-Degert type: 
A Richaud-Degert type is defined by
$$
d(n) = n^2 \pm r
$$
for $r|4n$ and $-n < r \le n$.
For $r$ fixed as above, the family $\{K_n\}$ of real quadratic fields is called R-D type.  
In this case, we have a bound of the regulator $R_{K_n}$:
$$
R_{K_n} < 3 \log \sqrt{D_{K_n}}
$$
As in imaginary quadratic case,  a well-known estimation of Siegel $L(1,\chi_D) \sim |D|^{-\epsilon}$
together with the class number formula
implies that there are only finitely many R-D type fields of class number one.
Assuming the generalized Riemann hypothesis, the class number one problems have been 
solved for many subfamilies in R-D type.

It is quite recent that Bir\'o  first obtained an Riemann hypothesis free answer to the class number one problem
for the families $K_n= \FQ(\sqrt{n^2+4})$ and  $K_n=\FQ(\sqrt{4n^2+1})$ in a series of papers(\cite{Biro1}, \cite{Biro2}). 
He investigated the behavior of the special values of the partial Hecke L-functions at $s=0$
in the family. The partial Hecke L-function of an ideal $\fa$ is defined for a ray class character $\chi$ as
$$
L(s,\fa,\chi):= \sum_{\fb\sim \fa} \frac{\chi(\fb)}{N\fb^s}.
$$
He discovered that the special values behave in a packet of linear forms
whose coefficients are easily computed for the family $(K_n, O_{K_n}, \chi_n := \chi\circ N_{K_n/\FQ})$ for a Dirichlet character $\chi$.  This property %of the L-values at $s=0$
is named the \textit{linearity}.

Inspired by Bir\'o's pioneering work, in \cite{Lee1}, \cite{Lee2}, \cite{Lee3} and \cite{Lee4}
the linearity is observed for more general families of Richaud-Degert types and 
the class number one and two problems have been answered for these.
% for more families of R-D types.

In \cite{J-L1}, we found  a sufficient condition to yield the linearity of the Hecke 
L-values at $s=0$. Namely, for families of  integral ideals $\{\fb_n$ in $K_n\}$ such that
$\fb_n^{-1} = [1,\omega(n)]:= \FZ 1 + \FZ \omega(n)$, where $\omega(n)$ has 
purely periodic positive continued fraction expansion of a fixed period $r$
$$
\omega(n) = [[a_0(n), a_1(n),\cdots, a_{r-1}(n)]]
$$
and $N(\fb_n)N(x\omega(n)+y) = b_0(n) x^2 + b_1(n) xy + b_2(n) y^2$
for $a_i(n)$  and $b_i(n)$ being integer coefficient linear forms in $n$.
In this setting we have,  for $n=qk+r$ with $0\le r < q$ and for a Dirichlet character $\chi$ of 
conductor $q$, the $L$-value at $s=0$ 
$$
L_{K_n}(0,\chi\circ N_{K_n}, \fb_n) = \frac{1}{12 q^2} \left(A_\chi (r) k + B_\chi (r)\right)
$$
with $A_\chi(r)$, $B_\chi(r) \in \FZ[\chi]$ where $\FZ[\chi]$ denotes the extension of $\FZ$ by the values of $\chi$.

In \cite{J-L2} we obtained a higher degree generalization  of the linearity for ray class partial zeta values.
Let $\omega(n)$ be the Gauss' reduced element of $\fb_n$.
If we allow the coefficient $a_i(n)$ of the continued fraction of $\omega(n)$ to be polynomial of degree $d$, then the partial zeta value at $s=0$ of a mod-$q$ ray class ideal $(C+D\omega(n))\fb_n$  in the class of $\fb_n$ % w.r.t. the conductor $q$
is a quasi-polynomial in $n$:
$$
\zeta_q(0,(C+D\omega(n))\fb_n) = \frac{1}{12 q^2} \left(A_0(r) + A_1(r) k + \cdots + A_d(r) k^d\right)
$$
with  $A_i(r)\in \FZ$(for precise definition, we refer the reader to \textit{loc.cit.}). 
In particular, if we take $d=1$ and sum the ray class zeta values twisted by $\chi_n$, %  for the ray classes in the ideal class of $\fb_n$ twisted by $\chi$, 
one can recover the linearity of the  partial Hecke L-values.
For $d >1$, the same process concludes % or has the same generalization to 
the polynomial behavior
of the partial Hecke values at $s=0$.

The purpose of this article is to generalize our earlier work to  special values at every nonpositive integer of the ideal class partial zeta functions
under the same assumption for the family $(K_n,\fb_n)$.
% of
%the partial zeta function of an ideal class of $\fb_n$ of a real quadratic field $K_n$ in family. 
We assume again $\fb_n^{-1} = [1,\omega(n)]$ where $\omega(n)$ has purely
periodic continued fraction expansion of fixed period $r$:
$$
\omega(n) = [[a_0(n), a_1(n),\cdots,a_{r-1}(n)]]
$$
such that and $N(\fb_n)N(x\omega(n)+y) = b_0(n) x^2 + b_1(n) xy + b_2(n) y^2$
for $a_i(n)$  and $b_i(n)$ being integer coefficient  polynomials. 
For the next two theorems, let $\ell$ be the even period of $\omega(n)$ (hence independent of $n$ and 
$\ell=2r$(reps. $r$) if $r$ is odd(reps. even)).

Our main result in this paper is as follows:
\begin{thm}\label{polynomial_behavior}
Let $N$ be a fixed subset of $\FN$.
Suppose $(K_n, \fb_n)$ satisfies the above condition for every $n\in N$.
Then the special value of the partial zeta function of $\fb_n$ at $s=-k$  for $k=0,1,2,\cdots$, is 
given by a polynomial in $n$:
$$
\zeta_{K_n}(-k,\fb_n) = A_0 + A_1 n + A_2 n^2 +  \ldots + A_m n^m
$$
of degree bounded by $m=kC+D$ with the coefficients $A_i \in \frac{1}{C_k}\FZ$, 
where $C$, $D$ and $C_k$ are given as follows:
\begin{equation*}
\begin{split}
C&=2 \deg \alpha_{\ell-1}+ \deg b_1 \\
%Max_{1\leq i\leq 3}\{deg A_i\} 
D&=\max_{0\leq i\leq r-1}\{\deg a_i\}
\end{split}
\end{equation*}
and 
\begin{equation*}
\begin{split}
C_k &= \text{LCM of }\\
& \Big\{ \text{the denominators of $\frac{B_{i+1}B_{2k+1-i}}{(i+1)(2k+1-i)}$ and $ \frac{B_{2k+2}}{(2k+2)(2k+1) }\begin{pmatrix}2k \\  i\end{pmatrix} ^{-1}$}\Big\}_{ 0\leq i \leq k}.
\end{split}\end{equation*}
\end{thm}

%for $k\ = 0,1,2, \ldots$. 
%\end{thm}

Our main theorem is a direct consequence of the following
estimation of the partial zeta values of an ideal $\fb$ in a real quadratic field $K$.

Let $\fb$ be an integral ideal such that $\fb^{-1}= [1,\omega]$ for $\omega>1$ and $0 < \omega' <1$. % be an ideal such that
Let $\alpha_i,\beta_i$ are coordinates of some lattice vectors determined by
(the continued fraction of) $\omega$. 
In particular, $\alpha_{\ell-1}\omega + \beta_{\ell-1}1$ 
is the totally positive fundamental unit
of $K=K_n$. See Sec.\ref{cfraction_decomp} for details. As usual, $B_i$ denotes the i-th Bernoulli number.

\begin{thm}\label{2nd_main}
Let $\fb$ be an ideal of a real quadratic field $K$ such that $\fb^{-1}=[1,\omega]$ where
$\omega = [[a_0, a_1,\ldots, a_{r-1}]]$. 
% and $\ell$ is even period.   
Then we have
\begin{equation*}
\begin{split}
&\zeta(-k,\fb) =\\
&\sum_{i=0}^{l-1}(-1)^{i-1}L_k(\partial_{h_1},\partial_{h_2})Q(\alpha_ih_1-\alpha_{i-1}h_2,\beta_i h_1-\beta_{i-1}h_2)^k
%\Big{|}
%_{h=(0,0)}
\\
&+\frac{B_{2k+2}}{(2k+2)!}\sum_{i=0}^{l-1}(-1)^ia_{\ell-i}R_k(\partial_{h_1},\partial_{h_2})Q(\alpha_{i-2}h_1+\alpha_{i}h_2,\beta_{i-2}h_1+\beta_{i}h_2)^k
%\Big{|}
%_{h=(0,0)}.
\end{split}
\end{equation*}
\end{thm}

In the above, $L_k$ and $R_k$ are the homogeneous polynomials of degree {2k}:
\begin{align}
L_k(X,Y)&=\sum_{i=1}^{2k+1}\frac{B_{i}}{i!}\frac{B_{2k+2-i}}{(2k+2-i)!}X^{i-1}Y^{2k-i+1},\\
R_k(X,Y)&=X^{2k}+X^{2k-1}Y+ \cdots+Y^{2k}.
\end{align}
%and $\alpha_i, \beta_i$ are nonnegative integers determined by the continued fraction of $\omega$(See Sec. \ref{cfraction_decomp}.). 

It is not surprising that this behavior of the partial zeta or L-values is related
to the pattern of the continued fractions in the family if we note that the Shintani cone decomposition 
arises in relation to the continued fraction.
The significance of this fact lies on that
for real quadratic fields, the regulator is controlled not only by the discriminant but also
by the period  $r$ of the positive continued fraction of the reduced element  $\omega$ in $K$. 
This is due to the following well-known upper bound:
% of the
%regulator as follows:
$$
R_K \le r \log \sqrt{D_K}.
$$
%where $r$ the period of the continued fraction of $\omega$. 

Study of special values of zeta or L-functions goes back to Euler. Euler evaluated $\zeta(-k)$ 
for $k=0,1,2,\ldots$, by using Euler-MacLaurin summation formula(cf. \cite{Car}). 
Later on, %through a series of papers, 
Siegel computed the values at nonpositive integers of  $\zeta_K(\fb,\ff,s)$ the ray class
partial zeta function for an ideal $\fb$ in a totally real number field $K$ w.r.t. a conductor $\ff$ 
based on the theory of modular forms(\cite{Siegel}). Shintani established a combinatorial description of the zeta values at nonpositive integers(\cite{Shin}). Beside the complicated contour integral, Shintani's method is a reminiscence of Euler's. %result. 
Similar approach was taken independently by Zagier in his evaluation of the partial zeta functions
of real quadratic fields at non-positive integers(\cite{zagier}). 
Actually, the Shintani's method has a strength over Siegel's that it is ready to use in 
p-adic interpolation in case of totally real fields via Cartier duality. 
This view
was clarified by Katz in \cite{katz_another}.

Our evaluation is along with the line of Shintani and Zagier. 
We will apply a version of Euler-Maclaurin summation 
formula due to Karshon-Sternberg-Weitsman(\cite{ka}). 
In \textit{loc.cit.}, they made a version of Euler-Maclaurin
formula taking care of the remainder term, so that one can apply this to expand  asymptotically a function given as summation of exponentials. Since one side of the Euler-Maclaurin formula is application of
appropriate version of Todd differential operator, the decomposition of Shintani cone is reflected 
additively due to the additivity of the Todd series under cone decomposition(See SS.\ref{2d-Euler-M}.). 

Similar computation was done by Garoufalidis-Pommersheim(\cite{Pom}). 
They applied the Euler-Maclaurin summation formula of Brion-Vergne(\cite{Brion}) to obtain the asymptotic expansion. 
Brion-Vergne's formula is exact summation on the lattice points inside  a simple polytope valid for
polynomials or polynomials in exponentials of linear forms. 
They took the Shintani cone as the cone over a lattice polytope and
varied the size inside the cone. 
In their treatment, the exact and the error terms are considered separately. The formula of Brion-Vergne is 
applied to the exact term and the error term is shown to be appropriately bounded.

%They treat the exact term and the error term separately.
%For the exact, Brion-Vergne's formula is applied and the error term is shown to be
%appropriately bounded. 

Our method differs from 
that of Garoufalidis and Pommersheim 
%\textit{loc.cit.}
 in two directions: 
First, we used positive continued fraction while they took negative continued fractions. Basically, via the transition formula between  positive and negative continued fractions, they contain more or less the same information of the ideal. 
% is much simpler than the negative continued fraction of the same number. 
But as is pointed at the beginning,  it is important to note that the period of the two continued fractions have no control on each other. 
As the regulator is concerned,
%Since the regulator is directly related
%to the positive continued fraction, 
it is better to express the zeta values 
using the terms of the positive continued fraction. 
Nevertheless, our earlier work has been made
via translation of the terms of the positive continued fraction into those of negative continued fraction. 
Hence the direct use of the positive continued fraction significantly reduces the amount of the computations needed.
While the negative continued fraction
yields an actual cone decomposition, the positive continued fraction gives rise to a virtual 
cone decomposition. 
In fact, the cone decomposition appeared at first to be a fan in toric geometry where no virtual
decomposition is allowed to define a toric variety. 
The Todd additivity can be simply extended to virtual decompositions if we take care of the orientation
of the cone and take it as the sign of the Todd series. 
%
%This is overcome by extending the Todd series to virtual cones via the
%additivity of the Todd series. 
%Certainly, this simplifies our earlier computation by a good amount. 
Second, we make a direct use of  the Euler-Maclaurin formula of Karshon-Sternberg-Weitsman.
Since the exponentials in our case is of Schwarz class toward the infinity of the cone to evaluate. 
Thus both sides of the Euler-Maclaurin formula make sense and we have more transparent proof.
%the proof is more transparent. 
In addition, Karshon-Sternberg-Weitman's version of Euler-Maclaurin has advantage over Brion-Vergne's
in that no limitation on the function to integrate while Brion-Vergne's, since the former can be applied to wider 
range of functions.
% formula strictly limited the case of
%functions. 

Again, some part of the computation is similar to what had been done by Zagier(\cite{zagier}).
He obtained the partial zeta values at non-positive integers again by  decomposition 
of the underlying cone of the zeta summation according to the negative continued fraction 
together with the Euler-Maclaurin formula. One should
note that this decomposition appears in the dual side while the additive decomposition of the Todd 
differential operator is taken in this paper and \cite{Pom}.
It was already pointed out in \cite{Pom} as ``M-additivity'' to ``N-additivity''. 
One could ask direct application of Zagier's method with the cone decomposition from the positive continued fraction. 
Unfortunately, in this setting, the both sides of Euler-Maclaurin formula don't make sense
but need certain renormalization process to avoid operations with infinity. 
In our setting, this is no problem as the domain of integration remains the same 
while the cone decomposition is reflected to the Todd differential operator.
% in the decomposition
%of the Todd series. 

The plan of this paper is as follows: First,  we rewrite the partial zeta function of an ideal as a zeta function of a quadratic form
weighted by a fundamental lattice cone of the Shintani decomposition(Sec. \ref{partial_zeta_cone}). 
We recall a standard asymptotic method to evaluate the zeta values at nonpositive integers and rebuild
a version of Euler-Maclaurin formula (Sec.\ref{asymptotics}-\ref{2d_EM}). 
Then we apply this Euler-Maclaurin formula to obtain an expression of the zeta values and 
another expression after the cone decomposition arising from the positive continued fractions(Sec.\ref{Additivity_Todd}-\ref{special_zeta_values}). 
The partial zeta values at $s=0,-1,-2$ are explicitly computed using our method and compared with previously known
results(Sec.\ref{computations}). 
Sec.\ref{vanishing_section}, which is technical and similar to the  computation by Zagier(\cite{zagier}), 
is devoted to the proof of vanishing of a part skipped in the previous sections. Finally, we apply this to some families of
real quadratic fields to prove our main theorem(Thm.\ref{polynomial_behavior}) and the polynomials are explicitly computed out for some families and for some small $k$(Sec.\ref{polynomial_family}).

\section{Partial zeta function of real quadratic fields}
%\section{Partial Zeta function and cone}
\label{partial_zeta_cone}

\subsection{Partial zeta function}

Let $K$ be a real quadratic field and $\fb$ be an ideal.
Through out this article, by partial zeta function, we mean the partial zeta function of an ideal class in narrow sense.
The partial zeta function of an ideal $\fb$ is defined as
%We consider the partial zeta function of $\fb$:
 $$\zeta(s,\fb):=\sum_{\substack{\fa\sim\fb\\ \fa: integral}}N(\fa)^{-s}$$
where $\fa\sim\fb$ means $\fb=\alpha\fa$ for totally positive $\alpha$ in $K$. 
This infinite series defines a holomorphic
function in the region $\Re (s)>1$ of the complex plane and has a meromophic continuation to the entire complex plane.
Since for an integral ideal $\fa$ in the narrow class of $\fb$ there exists totally positive element $a \in \fb^{-1}$ such that $\fa= a\fb$ and
vice versa, we can write again
$$
\zeta(s,\fb) = \sum_{[a] \in (\fb^{-1})^+/E^+} N(a\fb)^{-s},
$$
where $(\fb^{-1})^+$ denotes the set of totally positive elements of $\fb^{-1}$ and $E^+=E^+_K$ denotes the group of 
totally positive units of $K$.
Now we are going to describe the summation as taken inside the Minkowski space of $K$. % of $K_\R$.
Let  $(\iota_1,\iota_2 )$ be two real embeddings of $K$. Let us denote the Minkowski space of $K$ by
$$
K_\FR = K \otimes_{\FQ} \FR = K_{\iota_1} \times K_{\iota_2}
$$
Then one can identify an ideal $\fc$  with a lattice of $K_\FR$ given by its image under the diagonal embedding of $K$ into $K_\FR$:
$$
\iota = (\iota_1,\iota_2): K \to K_\FR,\quad(\iota(a)\mapsto (\iota_1(a),\iota_2(a))
$$
This is a full lattice in the Minkowski space.
%$K_\FR = K \otimes \FR = K_{\iota_1} \times K_{\iota_2}$ the Minkowski space of $K$.

%Let us denote
% the group of totally positive units by $E_K^+$ and 
%by $\epsilon$.
$E_K^+$  acts on the 1st quadrant of $K_\FR$
by coordinate-wise multiplication after the diagonal embedding.
Let $\epsilon$ be the totally positive fundamental unit of $K$.
A fundamental domain of this action
is given as a half-open cone $F_K$ of $K_\FR$ with basis $\{\iota(1),\iota(\epsilon)\}$:
%Let us denote this fundamental domain by
$$
F_K = \{ x \iota(1) + y \iota(\epsilon) \in \FR^2 | x \ge 0, y >0\}.
$$
%Let us denote this cone by $F_K$.
For an ideal $\fa$ of $K$ or a lattice $\Lambda$ of $K_\FR$, we denote its intersection with $F_K$  by $F_K(\fa)$ or $F_K(\Lambda)$, respectively.

%When we describe the partial zeta summation running over
For $[a]\in (\fb^{-1})^+/E^+$, there is a unique representative $a$ %can be uniquely
chosen
in $F_K(\fb^{-1})$.
Thus we have
$$
\zeta(s,\fb) = \sum_{a \in F_K(\fb^{-1})} N(a\fb)^{-s}.
$$

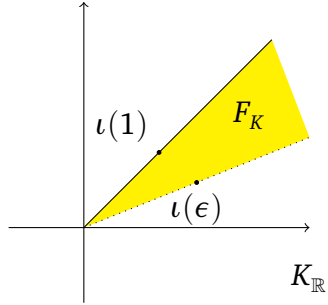
\begin{figure}\label{fundamental_cone}
\begin{tikzpicture}
\draw [->] (-1,0) -- (3,0);
\draw [->] (0,-1) -- (0,3);
\path [fill=yellow] (0,0) -- (2.5,2.5) --(3,1.2) -- (0,0);
\draw (0,0) -- (2.5,2.5);
\draw [fill] (1,1) circle [radius=0.025];
\node [above left] at (1,1) {$\iota(1)$};
\draw [fill] (1.5,0.6) circle [radius=0.025];
\node [below] at (1.5,0.6) {$\iota(\epsilon)$};
\node [above] at (3,-1) {$K_\FR$};
\draw [dotted] (0,0) -- (3,1.2);
\node at (2.2,1.5) {$F_K$};
\end{tikzpicture}
\caption{A fundamental cone of $E_+$-action}
\end{figure}

\subsection{Zeta function of 2-dimensional cones}
% of an ideal}
Consider the standard lattice $M :=\FZ^2$ in $\FR^2$.
Let $Q(x,y)=a x^2 + b xy + c y^2$ be a quadratic form.
For two linearly independent vectors $v_1, v_2$, let $\sigma(v_1,v_2)$ be the  cone in $\FR^2$ %defined by two linearly independent vectors $v_1,v_2$
as
the convex hull of the two rays $\FR^+ v_1$, $\FR^+ v_2$:
$$\sigma(v_1,v_2) := \{x_1v_1+x_2v_2\,|\, x_i>0 \text{ for } i=1,2\}.$$
For simplicity, we write $\sigma$ instead of $\sigma(v_1,v_2)$ if $v_1,v_2$ is clear from the context.
Following our convention on cones, the origin is not contained in $\sigma$.

% by two primitive lattice points $v,w$ in $\Lambda$.

%Define a weight function associated to $\sigma$ on $\Lambda$:
%We denote a 2-dimensional cone by $$\sigma=<(a,b),(c,d)>,$$
%where $(a,b)$ and $(c,d)$ is the smallest lattice point except $(0,0)$ in two lays of $\sigma.$
%And

Define a weight function $wt_{\sigma}$ with respect to $\sigma$ as follows:
\begin{equation}\label{zwt}
wt_{\sigma}(\ell)=
\begin{cases}1 & \ell\in int(\sigma)\\
\frac{1}{2} & \ell \in \partial(\sigma) -{(0,0)}\\
0 & \text{otherwise}
\end{cases}
\end{equation}

This strange weight is justified via identification of the partial zeta function with  the zeta function of a lattice cone that will be defined soon below.
The partial zeta function is a sum over the points of $F_K(\fb^{-1})$. Since the two edges of $F_K$ are related by the multiplication
of the totally positive unit, the summands over both edges coincide. When we take $F_K$ as half-open cone, this repetition is automatically removed. Equivalently, we may apply this weight function so that the total contribution over an orbit equals $1$.
The choice of assigning $1/2$ to each edge will be found useful when we apply the Euler-Maclaurin formula to the cone.

For $\sigma$ and  a quadratic form $Q(-)$ satisfying $Q(v) >0$ for $v\in \sigma$, we define
a zeta function as
the following series on $\Re(s)>1$:
$$
\zeta_Q(s,\sigma) := \sum_{m \in M} \frac{wt_\sigma(m)}{Q(m)^s}.
$$
%where $Q(l) = Q(x,y)$ for $l=(x,y)$.

\subsection{Comparison of zeta functions}

One can choose $\fb$ as an integral ideal in the same class such that
$\fb^{-1}=[1,\omega]$ for $[1,\omega]$ being the free $\FZ$-module generated by $1$ and $\omega$.
%Via $(\iota_1,\iota_2 ) $ two real embeddings of $K$, one considers $\fb$ as a full lattice in
%$K_\FR = K \otimes \FR = K_{\iota_1} \times K_{\iota_2}$ the Minkowski space of $K$.
Taking $\iota(1),\iota(\omega)$ as basis of $K_\FR$, we have trivialization
$$
K_\FR \simeq \FR^2\quad \text{and}\quad \iota(\fb^{-1}) \simeq M=\FZ^2.
$$
Here, we fix the order of the basis such that $x+y \omega$ reads $(y,x)$ in $\FR^2$.

From the reduction theory of quadratic forms, we have 
%there is 
a privileged choice of  $\omega$ such that
$$
\iota_1(\omega) >1, -1<\iota_2(\omega) <0.
$$
Then the totally positive fundamental unit $\epsilon$ belongs to $\fb^{-1}$ and $\epsilon = p + q\omega$ for a pair $(p,q)$ of relatively prime positive integers.

Let $\sigma$ be  lattice cone generated by $(0,1)$ and $(q,p)$, which corresponds to $F_K$.
One should be aware that this identification depends on $\fb$ and the choice of $\omega$.
Then we have for an integral ideal $\fa = a \fb$ with $a$ totally positive in $K$% for $a \in F_K(\fb^{-1})$
%is equal to
$$N(\fa)= N(m\omega+n)N(\fb)=Q(m,n),$$
for $a = m \omega +n\in F_K(\fb^{-1})$.
Thus we have the following identification of zeta functions:
\begin{lem}\label{zeta}
Let $Q(m,n) = N(m\omega+n)N(\fb) $ and $\sigma$ be a cone defined as above.
Then we have
$$
\zeta(s,\fb) =\zeta_Q (s,\sigma).
$$
%\sum_{(m,n)\in \Lambda}\frac{wt_{\sigma(\fb)}(m,n)}{Q(m,n)^s}.$$
%for $Q(m,n) = N(m\omega+n)N(\fb) $.
\end{lem}

\section{Euler-Maclaurin formula and Zagier's asymptotics}\label{asymptotics}

In the previous section, we have identified the partial zeta function of an ideal as a zeta
function of a quadratic form defined over a lattice cone.

%Now we come to the evaluation of the zeta function at non-positive integers. 
To evaluate the values at nonpositive integers, we will apply
Zagier's asymptotic method to the exponential series associated to the zeta function of quadratic
form running over the lattice points of the considered cone.
The  coefficients of the asymptotic expansion of the exponential series which will be  obtained 
via Euler-Maclaurin formula are the zeta values at nonpositive integers up to some simple factors.

%for the summation over lattice point of the cone.

We first recall the asymptotic method of Zagier then 
state  the appropriate Euler-Maclaurin formula for our case. 
%for a 2-dimensional lattice cone.

\subsection{Zagier's asymptotic method}

For a Dirichelet series of the following form
$$
\zeta_\Delta(s) := \sum_\lambda \frac{a_\lambda}{\lambda^s}, \quad\ \{ \lambda \} \subset \FR^+,\quad\ \lambda \rightarrow \infty,
$$
if meromophically continued to the entire complex plane,
there is a fairly standard approach to evaluate the values at nonpositive integers.

If $\sum_\lambda a_\lambda e^{-\lambda t}$ has asymptotic expansion $\sum_{i=-1}^{\infty} c_it^i$ at $t=0$,
then $ \zeta_\Delta(s)$ has meromorphic continuation to entire complex plane and
$$
\zeta_{\Delta}(-n) = (-1)^n n! c_n
$$
for a nonnegative integer $n$(See Prop.2 in \cite{zagier}).

%Therefore the special value of the partial zeta function at non-positive integer is obtained by
%asymptotic expression of the exponential sum of the associated cone zeta function.

Let $Q(-)$ be a quadratic form and $\sigma$ be a lattice cone such that $Q|_{\sigma}$ is positive.
The Dirichlet series defining the zeta function of $(Q,\sigma)$ yields the following exponential sum
\begin{equation}
E(t,Q,\sigma) := \sum_{\ell \in \sigma\cap M}  wt_\sigma(\ell) e^{-Q(\ell) t},
\end{equation}
where $wt(-)$ is the weight function defined for $\sigma$ in Sec. 2.

%Later we will see that
%$E(t,Q,\sigma)$ is meromorphic at $t=0$ and allows asymptotic expansion in $t$:
%\begin{equation}
%\begin{split}
%E(t,Q,\sigma) 
%Todd_{\check{\sigma}}(\frac{\partial}{\partial h_1},\frac{\partial}{\partial h_2}    )\int_{\sigma(h_1,h_2)} e^{-Q(x) t} dx \Big|_{(h_1,h_2)=0}
%\sim c_{-1} t^{-1} + c_0 + c_2 t + \cdots
%\end{split}
%\end{equation}

The asymptotic expansion of the above exponential series  will be computed after we 
state the appropriate version of Euler-Maclaurin formula with remainder for $\sigma$ and the weight 
in consideration. 
%, which will be presented 
%soon below. A direct application of Zagier's method shows that
%$$
%\zeta(\fb, -n+1) = (-1)^n n! c_{n}
%$$ 

\section{Euler-Maclaurin formula for 2-d cones}\label{2d_EM}

\subsection{Twisted Todd and L-series}

Let $\lambda$ be an $N$-th root of $1$.
%For a root of unity $\lambda$, 
We define a $\lambda$-twist of the classical Todd series.
$$
\td^\lambda(S) =  \frac{S}{1-\lambda e^{-S}}
$$
When $\lambda=1$, this is the classical Todd series.
This version of Todd series is used in \cite{Brion} where they take  the sum 
of the values of a function at the lattice points  (strictly) inside of a simple lattice polytope.  

As we use weight $1/2$ on the boundary rays of a cone, 
%we need to introduce 
another 
variant of the Todd series with $\lambda$-twist is defined as 
follows: for $\lambda$ a root of unity, 
%we set
$$
L^\lambda(S) = \frac{S}2 \frac{1+ \lambda e^{-\lambda}}{1- \lambda e^{-\lambda}}
= \frac{S}{1-\lambda e^{-S}} - \frac{S}2.
$$
We call this the $\lambda$-twisted $L$-series.
This fits well to the case when we put the weight $(1/2)^{cod}$ on a face, where `\textit{cod}'
denotes the codimension of the face. 
This is used in \cite{ka} in their version of Euler-Maclaurin formula. 
For $\lambda=1$, $L^1(S)$ is nothing but the even part of $\td(S)$.
Similarly to the case of Todd series, the series expansion at $S=0$ of $L^\lambda(S)$ is given as
$$
L^\lambda(S) = (\frac12 - \frac{\lambda}{1-\lambda}
)S+
Q_{2,\lambda}(0)S^2+Q_{3,\lambda}(0)S^3+\cdots+Q_{k,\lambda}(0)S^k + \cdots,
$$
where $Q_{i,\lambda}(x)$ is a (generalized) function of period $N$ for $i=0,1,2,\cdots$ 
defined as follows.

For $m=0$,
\begin{equation}
Q_{0,\lambda}(x)
:=-\sum_{n\in\FZ}\lambda^n\delta(x-n).
\end{equation}
 For $m > 1$,
$Q_{m,\lambda}(x)$ is an indefinite integral of
$Q_{m-1,\lambda}(x)$
with an integral constant fixed by the
boundary value condition
\begin{equation}
\int_0^{N}Q_{m,\lambda}(x)dx= Q_{m+1,\lambda}(N) - Q_{m+1,\lambda}(0) =0.
\end{equation}
Thus we have
$$\frac{d}{dx}Q_{m,\lambda}(x)=Q_{m-1,\lambda}(x)\text{\quad and\quad   }
\int_0^{N}Q_{m,\lambda}(x)dx=0.$$

Note that we are taking these $Q_{m,\lambda}$ for $m \geq 0$ as
distributions on $L^1_c([0,+\infty))$. $Q_{1,\lambda}$ is continuous and $Q_{m,\lambda}$ is $C^{m-1}$-function for $m\ge 2$.
These generalize  the periodic Bernoulli functions appearing in some literatures on analytic continuation
of the Riemann zeta function using the Euler-Maclaurin formula(eg. \cite{Car}).
% definition 
%of the Dedekind sums.

\subsection{Todd series of 2-dimensional cone}\label{2-d_Todd}

Let $M=\FZ^2\subset \FR^2$ be a fixed lattice.
Recall that a {\it lattice cone} is the convex hull of two rays generated by lattice vector.
We may assume the generating vectors of a cone are primitive(i.e. not a multiple of other lattice
vector in the same ray). For two linearly independent primitive lattice vectors $v_1, v_2$, let
$\sigma(v_1,v_2)$ be the cone generated by $v_1$ and $v_2$. 
When %In case 
$v_1,v_2$ are clear from
the context, we will simply write $\sigma$ intend of  $\sigma(v_1,v_2)$.
When there appear several cones, they will be denoted by $\sigma, \tau\ldots$ or $\sigma_1,\sigma_2,\sigma_3,\ldots$. % to distinguish them.
Since we will be concerned with surface integral over a 2-dimensional cone the order of the basis
vectors (ie. the orientation of the cone) is important. So  $\sigma(v_1,v_2)$ is never equal to $\sigma(v_2,v_1)$. %in general.
Taking $v_1,v_2$ as column vectors in $\FZ^2$, 
we associate a nonsingular $(2\times 2)$-matrix
$$
A_\sigma = (v_1, v_2)
$$
to a lattice cone $\sigma=\sigma(v_1,v_2)$. 
Conversely, if a $(2\times2)$-nonsingular matrix $A$ with integer coefficient has
column vectors $v_1,v_2$ which are primitive, we can associate a unique lattice cone.
%If $\det A_\sigma >0$, we say that $\sigma$ is a {\it positive cone} or {\it positively oriented}. Otherwise,
%$\sigma$ is called a {\it negative cone}.
A cone is said to be {\it nonsingular} if the matrix is in $GL_2(\FZ)$. Equivalently, $\sigma$ is nonsingular iff $\det(A_\sigma)=\pm 1$.

\begin{remark}
In literatures on polytopes or toric geometry, a cone is said to be simple if it is generated by $n$-linearly 
independent rays in $\FR^n$. In this article, as we are considering only 2 dimensional cones, every cone
is simple unless degenerate. 
\end{remark}

Let $M_\sigma$ be the sublattice of $M$ generated by $v_1,v_2$ and $\Gamma_\sigma$ be $M/M_\sigma$.
An element $g\in \Gamma_\sigma$ can be written as
$$
g = a_{\sigma,1}(g) v_1 + a_{\sigma,2}(g) v_2
$$
for rational numbers $a_{\sigma,1}(g), a_{\sigma,2}(g)$  modulo $\FZ$.
This is given ambiguously but yields two well-defined characters
$$
\chi_{\sigma,i} : g\mapsto e^{2\pi i a_{\sigma,i}(g)},\quad\text{for $i=1,2$.}
$$

The Todd power series for a cone $\sigma$ is defined as
$$
\td_{\sigma}(x_1,x_2):=\sum_{g\in \Gamma_\sigma}
\td^{\chi_{\sigma,1}(g)} (x_1) \td^{\chi_{\sigma,2}(g)} (x_2).
%\frac{x_1}{1-\chi_{\sigma,1}(g)e^{-x_1}}\frac{x_2}{1-\chi_{\sigma,2}(g)e^{-x_2}}
$$

Similarly, we define the L-series for $\sigma$ as
$$
L_\sigma(x_1,x_2) := \sum_{g\in \Gamma_\sigma}
L^{\chi_{\sigma,1}(g)} (x_1) L^{\chi_{\sigma,2}(g)} (x_2).
$$

These are used in the Euler-Maclaurin formula of \cite{Brion} and \cite{ka}, respectively.
For a 2-dim cone $\sigma$, the Todd and the L-series are related in the following manner:

\begin{lem}\label{tod}
%Let %$\lambda_{\gamma,i} = e^{2\pi i <\alpha_i,\gamma>}$ and  
%$|\Gamma_{\sigma}|=q$. Then 
%We have
$$
%\sum_{\gamma\in \Gamma_{\check{\sigma}}}\Big{(}\prod_{i=1}^2L^{\lambda_{\gamma,i}}(x_i)\Big{)}
L_\sigma(x_1,x_2)=\td_{\sigma}(x_1,x_2)^{\text{even}}-\frac{|\Gamma_\sigma|}4x_1x_2,$$
where $\td_\sigma(x_1,x_2)^{even}$ denotes the even part of $\td_\sigma(x_1,x_2)$ 
under $(x_1,x_2)\mapsto (-x_1,-x_2)$.
\end{lem}
\begin{proof}
%We note that for $\lambda\not=1$
From the definition of $L^\lambda(S)$, we have
$$L^{\lambda}(S)= L^{\lambda^{-1}}(-S).
%\frac{S}{2}\frac{1+\lambda e^{-S}}{1-\lambda e^{-S}}=\frac{S}{1-\lambda e^{-S}}-\frac{S}{2}.
$$
%$$
%Todd_{\sigma}(x_1,x_2)=\sum_{\gamma\in\Gamma_{\check{\sigma}}}\prod_{i=1}^2\Big{(}%L^{\lambda_{\gamma,i}}(x_i)+\frac{x_i}{2}\Big{)}
%$$
Let $\lambda_{\gamma,i} := e^{2\pi i \left<\alpha_i,\gamma\right>}= \chi_{\sigma,i}(\gamma)$ for $i=1,2$.
Then we have
\begin{equation*}
\begin{split}
2\td^{\text{even}}_{\sigma}(x_1,x_2)=& 
 \td_{\sigma}(x_1,x_2)+\td_{\sigma}(-x_1,-x_2)\\
=&\sum_{\gamma\in\Gamma_{\check{\sigma}}}\prod_{i=1}^2\Big{(}L^{\lambda_{\gamma,i}}(x_i)+\frac{x_i}{2}\Big{)}+\sum_{\gamma\in\Gamma_{\check{\sigma}}}\prod_{i=1}^2\Big{(}L^{\lambda_{\gamma,i}}(-x_i)-\frac{x_i}{2}\Big{)} \\
 = & 2 L_\sigma (x_1,x_2) + \frac{|\Gamma_\sigma|}2 x_1 x_2
\end{split}
\end{equation*}
This finishes the proof.
%Since $L^{\lambda}(-S)=L^{\lambda^{-1}}(S)$, we find that the above is equal to
%$$\sum_{\gamma\in\Gamma_{\check{\sigma}}}\prod_{i=1}^2\Big{(}L^{\lambda_{\gamma,i}}(x_i)+\frac{x_i}{2}\Big{)}+\sum_{\gamma\in\Gamma_{\check{\sigma}}}\prod_{i=1}^2\Big{(}L^{\lambda_{\gamma,i}^{-1}}(x_i)-\frac{x_i}{2}\Big{)}$$
%Moreover, we find that
%$$\sum_{\gamma\in\Gamma_{\check{\sigma}}}\prod_{i=1}^2\Big{(}L^{\lambda_{\gamma,i}}(x_i)\Big{)}=
%\sum_{\gamma\in\Gamma_{\check{\sigma}}}\prod_{i=1}^2\Big{(}L^{\lambda_{\gamma,i}^{-1}}(x_i)\Big{)}.$$
%And for $i=1,2$
%$$\sum_{\gamma\in\Gamma_{\check{\sigma}}}L^{\lambda_{\gamma,i}}(x_i)=
%\sum_{\gamma\in\Gamma_{\check{\sigma}}}L^{\lambda_{\gamma,i}^{-1}}(x_i).$$
%Thus

%\begin{equation*}
%\begin{split}
%&\sum_{\gamma\in\Gamma_{\check{\sigma}}}\prod_{i=1}^2\Big{(}L^{\lambda_{\gamma,i}}(x_i)+\frac{x_i}{2}\Big{)}+\sum_{\gamma\in\Gamma_{\check{\sigma}}}\prod_{i=1}^2\Big{(}L^{\lambda_{\gamma,i}^{-1}}(x_i)-\frac{x_i}{2}\Big{)}\\
%&=2\sum_{\gamma\in\Gamma_{\check{\sigma}}}\prod_{i=1}^2\Big{(}L^{\lambda_{\gamma,i}}(x_i)\Big{)}+\frac{qx_1x_2}{2}
%\end{split}
%\end{equation*}
\end{proof}

We say two cones $\sigma_1$ %=\sigma(v_1,v_2)$
and $\sigma_2$
%=\sigma(w_1,w_2)$
are {\it similar} 
%(resp. strictly similar)} 
if
$$A A_{\sigma_1} %(v_1,v_2)
= A_{\sigma_2}
$$
%(w_1,w_2) \text{
for $ A\in GL_2(\FZ)$.
% (resp. for $A\in SL_2(\FZ)$ ).
%}.$$
In this case, $A$ induces an isomorphism of $M_{\sigma_1}$ in $M_{\sigma_2}$,
which descends to isomorphism of $\Gamma_{\sigma_1}$ in $\Gamma_{\sigma_2}$.
Since this isomorphism takes the lattice generators of $\sigma_1$ to those of $\sigma_2$,
the two characters are preserved.  
\textit{A priori} the Todd series of two cones coincide. %we have the following:
\begin{prop}\label{sim}
For two similar cones $\sigma$ and $\tau$, we have
$$
\td_{\sigma}(x_1,x_2)=\td_{\tau}(x_1,x_2).
$$
\end{prop}
\begin{proof}
Clear.
\end{proof}
One should be aware that this is not the similarity of the matrices in linear algebra. 
$\sigma(v_1,v_2)$ and $\sigma(v_2,v_1)$ are not similar in general.
%It is easy to see that $M_{\sigma_1} = M_{\sigma_2}$ for two similar cones.

%\begin{proof}
%Clear.
%First, note that $A$ descends to an isomorphism $A:\Gamma_\sigma\rightarrow \Gamma_\tau$.
%We also note that if $$g=a_{\sigma,1}(g)v_1+a_{\sigma,2}(g)v_2$$ then
%$$Ag=a_{\sigma,1}(g)Av_1+a_{\sigma,2}(g)Av_2=a_{\sigma,1}(g)w_1+a_{\sigma,2}(g)w_2.$$
%Thus we have
%$$a_{\tau,i}(Ag)=a_{\sigma,i}(g).$$
%Finally we have
%\begin{equation*}
%\begin{split}
%&Todd_{\tau}(x_1,x_2)=\sum_{h\in \Gamma_\tau}\frac{x_1}{1-\chi_{\tau,1}(h) e^{-x_1}}\frac{x_2}{1-\chi_{\tau,2}(h) e^{-x_2}}\\
%&= \sum_{g\in \Gamma_\sigma}\frac{x_1}{1-\chi_{\tau,1}(Ag) e^{-x_1}}\frac{x_2}{1-\chi_{\tau,2}(Ag) e^{-x_2}}\\
%&=\sum_{g\in \Gamma_\sigma}\frac{x_1}{1-\chi_{\sigma,1}(g) e^{-x_1}}\frac{x_2}{1-\chi_{\sigma,2}(g) e^{-x_2}}=Todd_{\sigma}(x_1,x_2)
%\end{split}
%\end{equation*}
%\end{proof}

%{\large\bf Here, we move the notion of dual lattice summation away}

\subsection{Dual cone and its lattice}

%We keep the notation in \ref{2-d_Todd}. 
Let $N:=\Hom(M,\FZ)$ be the dual lattice of $M$.
$N$ is a lattice in the vector space $N_\FR := N\otimes \FR$.
Using the standard inner product $\left<x,y\right>$ we will often 
identify $N$ and $M$.
Associated to a lattice cone $\sigma$ in $M$, its dual cone $\check{\sigma}$
is defined as
$$
\check{\sigma} := \{ y \in N_\FR-0 | \left< y,x \right> \ge 0 \}
$$
As the orientation is concerned, $\check{\sigma}$ is endowed with the orientation given by
transpose of the matrix of $\sigma$.
Notice that $\check{\sigma}$ is again a lattice cone generated by two primitive lattice vectors inward 
and normal to $\sigma$. 
To $\check\sigma$, there are two lattices naturally associated. 
$
M_{\check{\sigma}}
$
is the sublattice of $M=N$ 
generated by the primitive lattice vectors of $\check\sigma$. 
%w.r.t. the reference
%lattice $N=M$. 
Note that this coincides with the definition of $M_\sigma$ in \ref{2-d_Todd}.
$N_{\check\sigma} := \Hom ( M_\sigma,\FZ)$
is a lattice in $N_\FR$ generated by dual vectors $\alpha_1,\alpha_2$ 
to $v_1, v_2$ if $\sigma=\sigma(v_1,v_2)$ (ie. $
\left< \alpha_i, v_j \right> = \delta_{ij}
$).
These  are related by the following inclusion relation:
$$
M_{\check{\sigma}} \subset M = N = \FZ^2 \subset N_{\check\sigma}
$$

\subsection{Euler-Maclaurin formula for 2-d cones}
\label{2d-Euler-M}
Let $L^{k,\lambda}(S)$ be the truncation of $L^{\lambda}(S)$ below degree $k$:
$$
L^{k,\lambda}(S) = (\frac{1}{2}+\frac{\lambda}{1-\lambda})S
+Q_{2,\lambda}(0)S^2 +\cdots+ Q_{k,\lambda}(0) S^k.
$$

For a cone $\sigma$, define a weight function as
$$
wt^{KSW}_\sigma(x) =
\begin{cases}
1 &  \text{$x\in int(\sigma)$}, \\
1/2 &  \text{$x\in \partial\sigma-0$}, \\
1/4 &  \text{$x=0$},\\
0 & \text{otherwise}
\end{cases} 
$$
Then the Euler-MacLaurin formula (with remainder) \textit{\`a la} Karshon-Sternberg-Weitsman in \cite{ka} applied to $f(x)$ evaluates the summation of the values of $f(x)$ on the lattice points of $\sigma$ w.r.t. $wt^{KSW}_\sigma$.
% as follows:
\begin{thm}[Karshon-Sternberg-Weitsman\cite{ka}]\label{ltod}
Let $\sigma$ be a lattice cone such that
$$
\check{\sigma}=\check{\sigma}(u_1,u_2)\text{ for } u_i\in \FZ^2.
$$
Let $\alpha_i$ be a dual basis of $u_i$(ie. $\left<\alpha_i,u_j \right>=\delta_{ij}$)
and $\lambda_{\gamma,i}:=e^{2\pi i \left<\gamma, \alpha_i \right>}$ for $i,j=1,2$.
Let 
$$
\sigma(h)=\{ x\in\FR^2|\big{<}x,u_i\big{>}\geq -h_i, \quad \text{for $i=1,2$}\}
$$
for $h=(h_1,h_2)\in \FR^2$.
%Let $L^{k,\lambda}(S)$
%be a $k$-th Taylor polynomial of $L^{\lambda}(S)$ .
Suppose $f$ is a smooth integrable function on  $\sigma(h)$ for some $h\in (\FR_+)^2$. Assume further
that $f$ is rapidly decaying toward the infinity of $\sigma$. 
Then we have
$$\sum_{x\in \sigma\cap\FZ^2}wt^{KSW}_{\sigma}(x) f(x)=\sum_{\gamma\in \Gamma_{\check{\sigma}}}\Big{(}\prod_{i=1}^2L^{k,\lambda_{\gamma,i}}(\partial _{h_i})\Big{)}\int_{\sigma(h)}f(x) dx \Big{|}_{h=0}+R_k^{\sigma}(f).
$$
where
\begin{equation*}
\begin{split}
&R_k^{\sigma}(f)=\\
&\frac{1}{|\Gamma_{\check{\sigma}}|}\sum_{\gamma\in\Gamma_{\check{\sigma}}}\int_{0}^{\infty}\int_{0}^{\infty}Q_{k,\lambda_{\gamma,2}}(x_2)\partial_{x_2}^kQ_{k,\lambda_{\gamma,1}}(x_1)\partial_{x_1}^kf(x_1 \alpha_1+x_2 \alpha_2)dx_1dx_2\\
&+\frac{1}{|\Gamma_{\check{\sigma}}|}\sum_{\gamma\in\Gamma_{\check{\sigma}}}(-1)^{k-1}\int_{0}^{\infty}\int_{0}^{\infty}Q_{k,\lambda_{\gamma,2}}(x_2)\partial_{x_2}^kL^{k,\lambda_{\gamma,1}}(-\partial_{x_1})f(x_1 \alpha_1 +x_2 \alpha_2)dx_1dx_2\\
&+\frac{1}{|\Gamma_{\check{\sigma}}|}\sum_{\gamma\in\Gamma_{\check{\sigma}}}(-1)^{k-1}\int_{0}^{\infty}\int_{0}^{\infty}Q_{k,\lambda_{\gamma,1}}(x_1)L^{k,\lambda_{\gamma,2}}(-\partial_{x_2})\partial_{x_1}^kf(x_1 \alpha_1 + x_2 \alpha_2)dx_1dx_2.
\end{split}
\end{equation*}
\end{thm}

In the above theorem, the summation is weighted with $wt^{KSW}_\sigma$.
Now we come to the same theorem with the weight changed to $wt_\sigma$ in \eqref{zwt} 
for our purpose. 
%Here $w^2_\sigma$ is defined as follows:
%$$
%wt_\sigma^2(x) =
%\begin{cases}
%1 &  \text{$x\in int(\sigma)$}, \\
%1/2 &  \text{$x\in \partial\sigma-0$}, \\
%1/4 &  \text{$x=0$},\\
%0 & \text{otherwise}
%\end{cases} 
%$$

\begin{cor}\label{wt2} With the same notations as in Thm.\ref{ltod}, we have that
$$
\sum_{x\in \sigma\cap\FZ^2} wt_{\sigma}(x) f(x)=\Big[\sum_{\gamma\in \Gamma_{\check{\sigma}}}\Big{(}\prod_{i=1}^2L^{k,\lambda_{\gamma,i}}(\partial_{h_i})-\frac{1}{4}\partial_{h_1}\partial_{h_2}\Big{)}\Big]\int_{\sigma(h)}f(x) dx \Big{|}_{h=0}+R_k^{\sigma}(f).
$$
\end{cor}
\begin{proof}
This comes from the following observation:
\begin{equation*}
\begin{split}
f(0,0)&=\partial_{h_1}\partial_{h_2}\int_{-h_2}^{\infty}\int_{-h_1}^{\infty}f(\alpha_1x_1+\alpha_2x_2)dx_1dx_2\Big{|}_{h=0}\\
&=
|\Gamma_{\check\sigma}|\partial_{h_1}\partial_{h_2}\int_{\sigma(h)}f(y)dy\Big{|}_{h=0}.
\end{split}
\end{equation*}
and
$$
\sum_{x\in \sigma\cap\FZ^2}wt_{\sigma}(x) f(x)=\sum_{x\in \sigma\cap\FZ^2}wt^{KSW}_{\sigma}(x) f(x)-\frac{1}{4}f(0,0).
$$
\end{proof}

\subsection{Asymptotic expansion}

%\begin{lem}\label{bou}

For $\lambda=e^{2\pi i \frac{j}{N}}$ an $N$-th root of $1$ and $k\ge 2$,
from the periodicity one sees that $Q_{k,\lambda}(x)$ is a bounded continuous function.
Since we are assuming that $f$ decays rapidly toward $x,y\to +\infty$,
we have
$$
\int_{0}^{\infty}\int_{0}^{\infty}\partial _{x_1}^i\partial _{x_2}^j 
f(\alpha_1x_1+\alpha_2 x_2)dx_1dx_2
$$
is bounded for any $i,j\ge 0$.

Thus the summation of the values of $f(tx)$ over $x$ running over lattice points of $\sigma$ is expressed
in series occurring in Euler-MacLaurin formula involving the $L$-series.
% instead of the Todd series.

\begin{thm}\label{assy}
Under the same assumption on $\sigma$ as in Thm.\ref{ltod}, we have the asymptotic expansion
%Let $\sigma$ be a cone such that
%$$\check{\sigma}=<u_1,u_2>\text{ for } u_i\in \FZ^2$$
%and $\alpha_i$ be a dual basis of $u_i$.
%Let $\lambda_{\gamma,i}=e^{2\pi i <\gamma,\alpha_i>}$ for $i=1,2$ .
for $t\in\FR^+$ and $t\rightarrow0$% we have
$$
\sum_{x\in \sigma\cap\FZ^2}wt_{\sigma}(x) f(tx)\sim 
\big(L_{\check{\sigma}}(\partial_{h_1},\partial_{h_2}) - \frac{|\Gamma_\sigma|}4 \partial_{h_1} \partial_{h_2} \big)
%\sum_{\gamma\in \Gamma_{\check{\sigma}}}\Big{(}\prod_{i=1}^2L^{\lambda_{\gamma,i}}(\partial_{h_i})-\frac{1}{4}\partial_{h_1}\partial_{h_2}\Big{)}
\circ\int_{\sigma(h)}f(tx) dx \Big{|}_{h=0}.
$$
\end{thm}
\begin{proof}
%Let $L^{k,\lambda}(S)$ be a $k$-th Taylor polynomial of $L^{\lambda}(S)$.
We note that
$$\int_{\sigma(h)}f(tx)dx=t^{-2}\int_{\sigma(th)}f(x)dx.$$
%Thus from theorem 
Applying Thm. \ref{ltod},  for $t\in \FR^{+}$, we have
$$
\sum_{x\in \sigma\cap\FZ^2}
wt_{\sigma}(x) f(tx)=\sum_{\gamma\in \Gamma_{\check{\sigma}}}\Big{(}t^{-2}\prod_{i=1}^2L^{k,\lambda_{\gamma,i}}(t\frac{\partial}{\partial h_i})-\frac{1}{4}\partial_{h_1}\partial_{h_2}\Big{)}\int_{\sigma(h)}f(x) dx
 \Big|_{h=0}+R_k^{\sigma}(f)(t),
 $$
with
\begin{equation*}
\begin{split}
&R_k^{\sigma}(f)(t)\\
=&\frac{1}{|\Gamma|}\sum_{\gamma\in\Gamma}\int_{0}^{\infty}\int_{0}^{\infty}Q_{k,\lambda_{\gamma,2}}(x_2)Q_{k,\lambda_{\gamma,1}}(x_1)\partial_{x_2}^k\partial_{x_1}^kf(\alpha_1tx_1+\alpha_2tx_2)dx_1dx_2\\
&+\frac{1}{|\Gamma|}\sum_{\gamma\in\Gamma}(-1)^{k-1}\int_{0}^{\infty}\int_{0}^{\infty}Q_{k,\lambda_{\gamma,2}}(x_2)\partial_{x_2}^kL^{k,\lambda_{\gamma,1}}(-\partial_{x_1})f(\alpha_1tx_1+\alpha_2tx_2)dx_1dx_2\\
&+\frac{1}{|\Gamma|}\sum_{\gamma\in\Gamma}(-1)^{k-1}\int_{0}^{\infty}\int_{0}^{\infty}Q_{k,\lambda_{\gamma,1}}(x_1)L^{k,\lambda_{\gamma,2}}(-\partial_{x_2})\partial_{x_1}^kf(\alpha_1tx_1+\alpha_2tx_2)dx_1dx_2,
\end{split}
\end{equation*}
for $\Gamma=\Gamma_{\check{\sigma}}$.
%\begin{equation*}
%\begin{split}
%&=\frac{1}{|\Gamma|}\sum_{\gamma\in\Gamma}t^{2k-2}\int_{0}^{\infty}\int_{0}^{\infty}Q_{k,\lambda_{\gamma,2}}(\frac{y_2}{t})\partial_{y_2}^kQ_{k,\lambda_{\gamma,1}}(\frac{y_1}{t})\partial_{y_1}^kf(\alpha_1y_1+\alpha_2y_2)dy_1dy_2\\
%&+\frac{1}{|\Gamma|}\sum_{\gamma\in\Gamma}(-1)^{k-1}t^{k-2}\int_{0}^{\infty}\int_{0}^{\infty}Q_{k,\lambda_{\gamma,2}}(\frac{y_2}{t})\partial_{y_2}^kM^{k,\lambda_{\gamma,1}}(-t\partial_{y_1})f(\alpha_1y_1+\alpha_2y_2)dy_1dy_2\\
%&+\frac{1}{|\Gamma|}\sum_{\gamma\in\Gamma}(-1)^{k-1}t^{k-2}\int_{0}^{\infty}\int_{0}^{\infty}Q_{k,\lambda_{\gamma,1}}(\frac{y_1}{t})M^{k,\lambda_{\gamma,2}}(-t\partial_{y_2})\partial_{y_1}^kf(\alpha_1y_1+\alpha_2y_2)dy_1dy_2
%\end{split}
%\end{equation*}
%Let $L^{k,\lambda_{\gamma,i}}(x)=\sum_{r=1}^{k} a_r^{\gamma,i}x^r$ and  

Let us change the variables: $y_i=tx_i$ for $i=1,2$.
Then
from the boundedness of $Q_{k,\lambda_{\gamma,i}}$, as $f(x,y)$ decays rapidly, 
one can easily see for example 
a summand of $R^\sigma_k(f)(t)$ in the second line belongs to $O(t^{k-1})$ at $0$:
 %from lemma \ref{bou} and \ref{sch}, we obtain that
\begin{equation*}
\begin{split}
&\int_{0}^{\infty}\int_{0}^{\infty}Q_{k,\lambda_{\gamma,2}}(x_2)\partial_{x_2}^kL^{k,\lambda_{\gamma,1}}(-\partial_{x_1})f(\alpha_1tx_1+\alpha_2tx_2)dx_1dx_2\\
&=t^{k-2}\int_{0}^{\infty}\int_{0}^{\infty}Q_{k,\lambda_{\gamma,2}}(\frac{y_2}{t})\partial_{y_2}^kL^{k,\lambda_{\gamma,1}}(-t\partial_{y_1})f(\alpha_1y_1+\alpha_2y_2)dy_1dy_2\\
%&=t^{k-2}\int_{0}^{\infty}\int_{0}^{\infty}Q_{k,\lambda_{\gamma,2}}(\frac{y_2}{t})\partial_{y_2}^k%\sum_{r=1}^{k} a_r^{\gamma,i}(-t\partial_{y_1})^rf(\alpha_1y_1+\alpha_2y_2)dy_1dy_2\\
%&=t^{k-1}\int_{0}^{\infty}\int_{0}^{\infty}Q_{k,\lambda_{\gamma,2}}(\frac{y_2}{t})\partial_{y_2}^k%\sum_{r=1}^{k} a_r^{\gamma,i}t^{r-1}(-\partial_{y_1})^rf(\alpha_1y_1+\alpha_2y_2)dy_1dy_2\\
&=O(t^{k-1})
\end{split}
\end{equation*}
Similarly, one can show the rest belongs to $O(t^{k-1})$. Thus we conclude that
%\begin{equation*}
%\int_{0}^{\infty}\int_{0}^{\infty}Q_{k,\lambda_{\gamma,1}}(x_1)L^{k,\lambda_{\gamma,2}}(-\partial_{x_2})\partial_{x_1}^kf(\alpha_1tx_1+\alpha_2tx_2)dx_1dx_2
%=O(t^{k-1})
%\end{equation*}
%and
%$$
 %\int_{0}^{\infty}\int_{0}^{\infty}Q_{k,\lambda_{\gamma,2}}(x_2)Q_{k,\lambda_{\gamma,1}}(x_1)\partial_{x_2}^k\partial_{x_1}^kf(\alpha_1tx_1+\alpha_2tx_2)=O(t^{2k-2}).
 %$$
%Altogether, finally we conclude that
$$
R_k^{\sigma}(f)(t)=O(t^{k-1}).
$$
\end{proof}

The asymptotic expansion in theorem \ref{assy} can be rewritten using todd power series from lemma \ref{tod}.
\begin{thm}\label{Asy}
 For $t\in\FR^+$  and
$t\rightarrow 0$, we have the asymptotic expansion
$$
\sum_{x\in \sigma\cap\FZ^2}wt_{\sigma}(x) f(tx)\sim
\Big{(}\td_{\check{\sigma}}^{\text{even}}(\partial_{h_1},\partial_{h_2})-\frac{q}{2}\partial_{h_1}\partial_{h_2}\Big{)}\circ \int_{\sigma(h)}f(tx) dx \Big{|}_{h=0}.
$$
\end{thm}

%Using a version of Euler-Maclaurin formula due to Karshon-Sternberg-Weitsman (see \cite{ka}), 
%we obtain
%\begin{equation}
%\begin{split}
%E(t,Q,\sigma) &\sim Todd_{\check{\sigma}}(\frac{\partial}{\partial h_1},\frac{\partial}{\partial h_2}    )\int_{\sigma(h_1,h_2)} e^{-Q(x) t} dx \Big|_{(h_1,h_2)=0}\\
%& = c_{-1} t^{-1} + c_0 + c_2 t + \cdots
%\end{split}
%\end{equation}

%The coefficient $c_n$ of the asymptotic expansion of $E(t,Q,\sigma)$ is easily obtained by applying the degree $n$ homogeneous
%part of $Todd_{\check\sigma}(\frac{\partial}{\partial h_1},\frac{\partial}{\partial h_2})$.

%{\bf Karshon state ....}

\subsection{Evaluation of zeta values}
 
Now we apply this to the exponential series associated to the partial zeta function of an ideal. We saw that a partial zeta function can
be identified with a zeta function of a cone $\sigma$ w.r.t. a quadratic form $Q$.
We are going to apply Zagier's theorem  to evaluate the special values of $\zeta(\fb,s)$ at non positive integers.
%From Zagier's theorem, to evaluate the special values of $\zeta(\fb,s)$ at non positive integers,
%we need the following asymptotic expansion:
%\begin{prop}\label{half}
We apply Thm. \ref{Asy} to obtain the asymptotic expansion of the exponential series. 
We take $f(tx) = e^{-Q(t^{1/2}x_1, t^{1/2}x_2)} = e^{-Q(x_1,x_2)t}$.
Then we obtain the following asymptotic expansion:
\begin{equation*}
\begin{split}
&\sum_{l\in\sigma\cap M} wt_{\sigma}e^{-Q(l)t}\sim\\
&\Big{\{}(\td_{\check{\sigma}}^{even}(\partial_{h_1},\partial_{h_2})-\frac{q}{2}\partial_{h_1}\partial_{h_2}\Big{\}}\circ\int_{\sigma(h_1,h_2)}e^{-Q(x_1,x_2)t}dx_1dx_2\Big|_{(h_1,h_2)=0}
\end{split}
\end{equation*}
%where $$\delta_{n,0}=\begin{cases}1 &\text{if  $n=0$,}\\0&\text{ otherwise}\end{cases}$$ and
%$\alpha_{s-1}$ is defined in equation (\ref{albe}).
%\end{prop}
%\begin{proof}
%Let $f(x)=e^{-Q(x)}$. Then $$e^{-Q(x)t}=f(t^{\frac{1}{2}}x).$$
%By applying theorem \ref{Asy}, we obtain above result.
%\end{proof}

\begin{thm}[Garoufalidis-Pommersheim\cite{Pom}]\label{Garou-Pommer}
For $n\geq0,$ we have
\begin{equation*}
\begin{split}
&\zeta(\fb,-n)=\\
&(-1)^n n!\Big{\{}\td_{\check{\sigma}}^{(2n+2)}(\partial_{h_1},\partial_{h_2})-\delta_{n,0}\frac{q}{2}\partial_{h_1}\partial_{h_2}\Big{\}}\circ\int_{\sigma(h_1,h_2)}e^{-Q(x_1,x_2)}dx_1dx_2\Big|_{(h_1,h_2)=0}
\end{split}
\end{equation*}
where $$\delta_{n,0}=\begin{cases}1 &\text{if  $n=0$,}\\0&\text{ otherwise}\end{cases}$$
\end{thm}

\begin{remark}
In \cite{Pom}, they used the exact Euler-Maclaurin formula of Brion-Vergne(\cite{Brion}) to
a cone over a polytope.
We apply the version with remainder term due to Karshon-Sternberg-Weitsman in \cite{ka}. 
Since this is a
formula with remainder, we have a direct method to evaluate the zeta values at nonpositive
integers.
\end{remark}

\section{Additivity of Todd series and cone decomposition}\label{Additivity_Todd}
In this section, we recall some technic used in \cite{Pom} concerning  additive 
decomposition of the Todd series w.r.t. cone decomposition.

Todd series  does not allow decomposition in its original shape until
we normalize. 
The {\it normalized todd power series} $S_{\sigma}$ for a
cone $\sigma$ is defined as follows:
$$
S_{\sigma}(x_1,x_2)=\frac{1}{\det (A_\sigma)x_1x_2}\td_{\sigma}(x_1,x_2).
$$
%where $q$ is the number of set $\Gamma$
One should note that different choice of the orientation of the same
underlying cone yields the opposite sign in the normalized Todd
series and interchanges the two variables. This is contrary to the original Todd series case, where
the similarity class is determined by the sign.

Let  $v_i\in \FR^2$ for $i=1,2,3$ be pairwise linearly independent
primitive lattice vectors in a half-plane.
%such that  $$v_2\in C(v_1,v_3).$$
An ordered pair $(v_i,v_j)$ for $i\ne j$ determines a lattice
cone $\sigma_{ij}= \sigma_{ij}(v_i,v_j)$ with orientation.
%Then $\sigma(v_1, v_2)$ and $\sigma(v_2,v_3)$ are also cones in $\FR^2$.

In this case, we write formally
$$
\sigma_{ij} + \sigma_{jk}= \sigma_{ik}.
$$
Then we have the following:
\begin{thm}[Garoufalidis-Pommersheim \cite{Pom}] \label{add}
For $i=1,\cdots,r+1$, let $v_i$  be pairwise linearly independent lattice points in a half plane
of $\FR^2$. We define cones
$$
\sigma_i:=\sigma_i(v_i,v_{i+1}),\,\, \sigma:=\sigma(v_1,v_{r+1})
$$
Thus $$\sigma=\sigma_1+\sigma_2+ \cdots +\sigma_r.$$ 
Then
$$
S_{\sigma}(x_1,x_2)=\sum_{i=1}^{r}
S_{\sigma_i}(A_{\sigma_{i}}^{-1}A_\sigma (x_1,x_2)^{t}).
$$
In particular, if every  $\sigma_i$ is nonsingular(i.e. $\det(A_{\sigma_i}) =\pm1$) for
$i=1,2,\ldots,r$,
$$
S_{\sigma}(x_1,x_2)=\sum_{i=1}^{r}\det (A_{\sigma_i})
F(A_{\sigma_{i}}^{-1}A_{\sigma }(x_1,x_2)^{t}),
$$
where $F(x_1,x_2)=\frac{1}{1-e^{-x_1}}\frac{1}{1-e^{-x_2}}.$
\end{thm}
\begin{proof}
See Thm. 2 in \cite{Pom2}.
\end{proof}

\begin{remark} Abusing the notation, we denote $\sigma(v_2,v_1)$ by $-\sigma(v_1,v_2)$. 
Actually by definition of Todd power series of cone above, we easily find that %for primitive lattice points $v_1,v_2$,
$$\td_{\sigma}(x_1,x_2)=\td_{-\sigma}(x_2,x_1).$$

The matrix $A_\sigma^{-1}$ represents the linear transformation $v_1 \mapsto e_1$, $v_2 \mapsto e_2$. So we have $A_{-\sigma}^{-1}= \begin{pmatrix} 0 & 1 \\ 1 & 0 \end{pmatrix} A_\sigma$.
Let $A_\sigma^{-1} = \begin{pmatrix} w_1 \\ w_2 \end{pmatrix}$ for two row vectors $w_1$, $w_2$.
Then $A_{-\sigma}^{-1} = \begin{pmatrix} w_2 \\ w_1 \end{pmatrix}$.
Therefore,
\begin{equation}
\begin{split}
\td_{\sigma}(A_\sigma^{-1}(x_1,x_2)^t)&=\td_{\sigma}( \left<w_1, (x_1,x_2 )\right>, \left< w_2, (x_1,x_2) \right>) \\
&= \td_{-\sigma} ( \left< w_2, (x_1,x_2) \right>, \left<w_1, (x_1,x_2 )\right> )\\
&= \td_{-\sigma} (A_{-\sigma}^{-1} (x_1,x_2)^t). 
\end{split}
\end{equation}
Thus one can see easily that for  the  additivity theorem to hold the orientation of $\sigma$
does not make any problem.
\end{remark}

\section{Cone decomposition and Continued fraction}\label{cfraction_decomp}

In this section, we will decompose the cone $\sigma(\fb^{-1})$ into nonsingular cones.
This decomposition follows directly the decomposition of the fundamental cone in the 
totally positive quadrant of Minkowski space under the action of the totally positive unit group.  
This is fairly standard fact related to desingularization of a cusp of the Hilbert modular
surface of the real quadratic field considered. It is described in terms of the (minus) continued fraction
expansion of the reduced basis of $\fb^{-1}$ so that the desingularization of the lattice cone $\sigma(\fb^{-1})$ in the sense of toric geometry follows(cf. \cite{Fulton}, \cite{van}). 
We are going to apply Thm. \ref{add} to obtain explicit formula of the zeta values using the terms of the positive continued fraction. One should note that our expression is differed from \cite{Pom} in that we use
the positive continued fraction instead of the negative one.

In general, there are many other decompositions possible for a singular cone.
But those lattice cones arising from a singular cone but for a totally real field and the action of the totally positive units has a decomposition after the shape of its Klein polyhedron which is a geometric realization of 
a continued fraction. In 2 dimension, this appears as follows: 
In each quadrant of $K_\FR$, we take the convex hull of
$\fb^{-1}=[1,\omega]$ and union the polygonal hulls. This is
the \textit{Klein polyhedra} of the ideal lattice
$\fb^{-1}$. One can further assume $\omega$ is a reduced basis:$\omega>1$, $-1<\omega'<0$ for
$\omega\in K$. Then $\omega$ has purely periodic positive
continued fraction expansion
$$
\omega=[[a_0,a_{1},\cdots,a_{r-1}]].
$$

%{\bf Figure of continued fraction and B i }
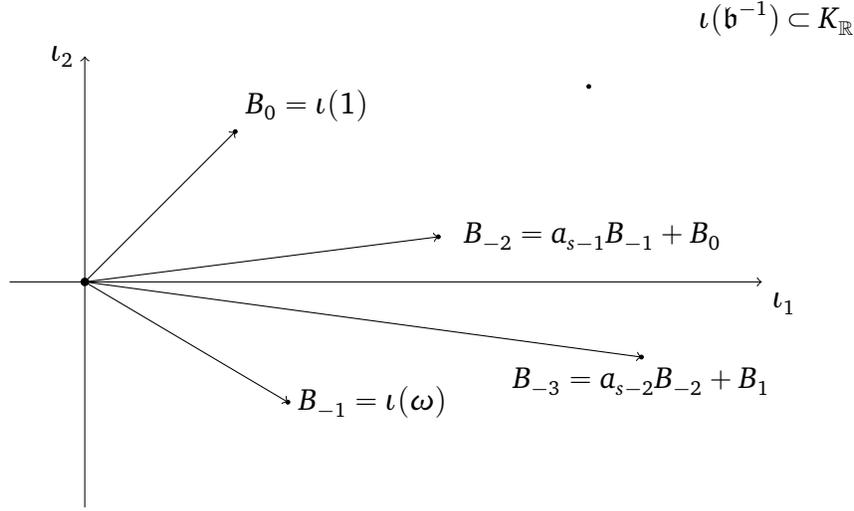
\begin{figure}\label{c-frac}
\begin{tikzpicture}
\draw [->] (-1,0) -- (9,0);
\node [below right] at (9,0) {$\iota_1$};
\node [left] at (0,3) {$\iota_2$};
\draw [->] (0,-3) -- (0,3);
%\path [fill=yellow] (0,0) -- (2.5,2.5) --(3,1.2) -- (0,0);
%\draw (0,0) -- (2.5,2.5);
\draw [fill] (2,2) circle [radius=0.025];
\draw [fill] (0,0) circle [radius=0.05];
\draw [fill] (4,4) circle [radius=0.025];
\draw [->] (0,0) -- (2.7,-1.6);
\draw [->] (0,0) -- (2,2);
\draw [->] (0,0) -- (4.7,0.6);
\draw [->] (0,0) -- (7.4, -1) ;
\node [below] at (7.4,-1) {$B_{-3}= a_{s-2} B_{-2} + B_1$};
\draw [fill] (7.4,-1) circle [radius=0.025];
\node [above right] at (2,2) {$B_{0}=\iota(1)$};
\draw [fill] (2.7,-1.6) circle [radius=0.025];
\node [right] at (2.7,-1.6) {$B_{-1} = \iota(\omega)$};
\draw [fill] (4.7,0.6) circle [radius=0.025];
\node [right] at (4.9,0.6) {$B_{-2}= a_{s-1} B_{-1} + B_0$};
\draw [fill] (6.7,2.6) circle [radius=0.025];
%\node [above] at (3,-1) {$K_\FR$};
%\draw [dotted] (0,0) -- (3,1.2);
\node at (9.2,3.5) {$\iota(\fb^{-1})\subset K_\FR$};
\end{tikzpicture}
\caption{$B_i$ and the continued fraction $[[a_0,a_1,\ldots,a_{r-1}]]$  }
\end{figure}

Let $\{B_{2i}\}$ (resp. $\{B_{2i+1}\}$) be the vertices of the convex hull of $\iota(\fb^{-1})$ in the 1st (resp. the 4th)  quadrant of $\FR^2$ 
%$\Bbb{R}^2_1$ (resp. in $\Bbb{R}^2_4$)
with $B_0=\iota(1), B_{-1}=\iota(\omega)$ and $x(B_i)<x(B_{i-1})$, where $x(-)$ is taking the 1st coordinate. These $B_i$ arising as the vertices of the Klein polyhedron should not be confused with the Bernoulli numbers.

Let $\ell$ be the even period of the continued fraction expansion of $\omega$(ie. $
\ell= r$ (resp. $2r$) for even  $r$ (resp. for odd $r$).

%The numbers $\alpha_i, \beta_i$ for $i=1,\ldots,s-1$ can be read
%of the continued fraction expansion of $\omega$ as follows:
%\begin{equation}\label{albe}
%\begin{pmatrix} \alpha_i \\ \beta_i \end{pmatrix} =
%\begin{pmatrix} a_{s-1} & 1 \\ 1 & 0 \end{pmatrix}
%\begin{pmatrix} a_{s-2} & 1 \\ 1 & 0 \end{pmatrix} \cdots \begin{pmatrix} a_{s-i} & 1 \\ 1 & 0 \end{pmatrix}
%\begin{pmatrix} 1 \\ 0 \end{pmatrix}
%\left(\begin{array}{c}\alpha_{i} \\\beta_i\end{array}\right) := \left(\begin{array}{cc}a_{s-1} & 1 \\1 & 0\end{array}\right)\left(\begin{array}{cc}a_{s-2} & 1 \\1 & 0\end{array}\right).....\left(\begin{array}{cc}a_{s-i} & 1 \\1 & 0\end{array}\right)\left(\begin{array}{c}1 \\0\end{array}\right).
%\end{equation}
%These are  primitive lattice vectors in $M$.

$B_i$ satisfies a periodic recursive relation read from the continued fraction of $\omega$ (cf. \cite{van}):
\begin{equation}
B_{i-1} = a_i B_i + B_{i+1}
\end{equation} 
Since a successive pair $B_{i}, B_{i+1}$ is a basis of the lattice $\iota(\fb^{-1})$ in $K_\FR$, 
this yields a change of basis
$$
\begin{pmatrix}
B_{i-1} & B_{i} 
\end{pmatrix}
= 
\begin{pmatrix}
B_i & B_{i+1}
\end{pmatrix}
\begin{pmatrix}
a_i & 1 \\ 1 & 0
\end{pmatrix}
$$

After successive change of basis, we have
\begin{equation}\label{albe}
(B_{i-1} ~ B_{i}) = (B_{-1}~ B_0) 
\begin{pmatrix}
a_{\ell-1} & 1 \\ 1 & 0
\end{pmatrix}
\begin{pmatrix}
a_{\ell-2} & 1 \\ 1 & 0
\end{pmatrix}\cdots\begin{pmatrix}
a_{\ell-i} & 1 \\ 1 & 0
\end{pmatrix}
\end{equation}

Let $\alpha_i, \beta_i$ be the coordinate of $B_{i-1}$ w.r.t. the basis $\{B_{-1}, B_0\}$:
$$
B_{i-1} = \alpha_i B_{-1} + \beta_i B_{0}
$$

Note that the coulumn vector $(\alpha_i , \beta_i)^t$ is equal to the 1st column of the matrix
$\begin{pmatrix}
a_{\ell-1} & 1 \\ 1 & 0
\end{pmatrix}
\begin{pmatrix}
a_{\ell-2} & 1 \\ 1 & 0
\end{pmatrix}\cdots\begin{pmatrix}
a_{\ell-i} & 1 \\ 1 & 0
\end{pmatrix}
$ in \eqref{albe}.

As $B_i$ are primitive, so is $(\alpha_i,\beta_i)$
% are primitive lattice vectors
in $M=\FZ^2$.

%Then
%$$B_{-s}=\iota(\epsilon)=\alpha_{s-1}B_{-1}+\beta_{s-1}B_0,$$
%where  $\epsilon>1$ is the totally positive fundamental unit in
%$K$.
%(See \ref{positive_continued_unit} in this section).

%Let $\{B_{2i}\}$ (resp. $\{B_{2i+1}\}$) be the vertices of the convex hull of $\iota(\fb^{-1})$ in
%$\Bbb{R}^2_1$ (resp. in $\Bbb{R}^2_4$)
%with $B_0=\iota(1), B_{-1}=\iota(\omega)$ and $x(B_i)<x(B_{i-1}).$
%Then
%$$B_{-s}=\iota(\epsilon)=\alpha_{s-1}B_{-1}+\beta_{s-1}B_0,$$
%where  $\epsilon>1$ is the totally positive fundamental unit in $K$.

%where  for $1\leq i\leq s-1,$

 % $(\alpha_i,\beta_i)$
%From the continued fraction expansion of $\omega$, we have  lattice points:

%Then with these notations, we have the following expression of the partial zeta function:

%To prove above proposition, we need the following lemmas:
%Let $\Bbb{R}^2_1$[resp. $\Bbb{R}^2_4$] be the first quadrant[resp. fourth quadrant]. A function
%$$x: \Bbb{R}^2\rightarrow \Bbb{R}$$
%is defined by $x$-coordinate of point in $\Bbb{R}^2.$

In the following, the totally positive fundamental lemma is identified:
%Then we have the following:
\begin{lem}\label{convex}
Let $\epsilon$ be the totally positive fundamental unit of $K$. Then
$$\iota(\epsilon)=
B_{-\ell}=\alpha_{\ell-1}B_{-1}+\beta_{\ell-1}B_0.$$
%where  $\epsilon>1$ is the totally positive fundamental unit in $K$.
%where
%$$(\alpha_{-1},\beta_{-1})=(0,1),\,\,(\alpha_0,\beta_0)=(1,0)$$
%and for $i\geq 1$, $(\alpha_i,\beta_i)$ are defined in Proposition \ref{zeta}.
\end{lem}
\begin{proof}
%We find the following equation in \cite{van}.
%$$B_{i-1}=a_i B_i+B_{i+1}.$$
Let $\epsilon_K>1$ be the fundamental unit of $K$.
% in other words, the unit group $E_K$ of $K$ is $\{\pm1\}\times <\epsilon_K>.$
Then for the period $r$ of continued fraction expansion of  $\omega$, we have
$$\iota(\epsilon_K)=B_{-r}.$$
See p.40 of \cite{van} for detail.
Since the totally positive unit $\epsilon$ is either $\epsilon_K$ or $\epsilon_K^2$ according to the sign of $\iota_2(\epsilon_K)$, 
we then obtain that $\iota(\epsilon)=B_{-\ell}$.
%
%If $\epsilon_K$ is totally positive then totally positive fundamental unit $\epsilon$ is equal to $\epsilon_K$ and  if not then we have
%$$\epsilon=\epsilon_K^2.$$
%Moreover we note that  for $i\in \FZ$ [See equation (5) in \cite{J-L}],
%$$\epsilon_K \circ B_i= B_{i-r}.$$
%Thus we find that
%$$\iota(\epsilon)=B_{-s}.$$
%Now we define
% $$(\alpha_{-1},\beta_{-1})=(0,1),\,\,(\alpha_0,\beta_0)=(1,0).$$
%Then we easily find that  for $i=0,1,2$
%$$B_{-i}=\alpha_{i-1}B_{-1}+\beta_{i-1}B_0.$$
%We assume  the for $i\leq k$
%$$B_{-i}=\alpha_{i-1}B_{-1}+\beta_{i-1}B_0.$$
%We also note that
%$$\left(\begin{array}{cc}\alpha_{i-1} & \alpha_{i-2} \\\beta_{i-1} & \beta_{i-2}\end{array}\right)\left(\begin{array}{cc}a_{-i} & 1 \\1 & 0\end{array}\right)\left(\begin{array}{c}1 \\0\end{array}\right)=\left(\begin{array}{c}\alpha_{i} \\\beta_{i}\end{array}\right). $$
%Thus from above assumption, we have
%\begin{equation}
%\begin{split}
%&B_{-k-1}=a_{-k}B_{-k}+B_{-k+1}\\
%&=a_{-k}(\alpha_{k-1}B_{-1}+\beta_{k-1}B_0)+\alpha_{k-2}B_{-1}+\beta_{k-2}B_0\\
%&=(a_{-k}\alpha_{k-1}+\alpha_{k-2})B_{-1}+(a_{-k}\beta_{k-1}+\beta_{k-2})B_0\\
%&= \alpha_k B_{-1}+\beta_k B_0.
%\end{split}
%\end{equation}
%Thus we have
%$$\iota(\epsilon)=B_{-s}=\alpha_{s-1}B_{-1}+\beta_{s-1}B_0$$
\end{proof}

Recall that we associated 
a lattice cone $\sigma(\fb^{-1})$ in $\FR^2$ to $\fb^{-1}$ in Sec. \ref{partial_zeta_cone}.
%is given as
%follows:
\begin{equation}\label{cone}
\sigma(\fb^{-1}) :=\sigma((0,1),(\alpha_{\ell-1},\beta_{\ell-1})).
\end{equation}
%where for $v_1,v_2\in \FR^2$ $$C(v_1,v_2)=\{x_1v_1+x_2v_2
%\,\,|\,\, x_i\geq 0\}.$$
This corresponds to the cone bounded by $\iota(1)$ and $\iota(\epsilon)$ in $K_{\FR}$.
%Through this section since we are only dealing with the partial zeta function of $\fb$. 
For the rest of this section,  only the cone $\sigma(\fb^{-1})$ is need to consider  to compute the zeta values. So we will write simply $\sigma$ instead of $\sigma(\fb^{-1})$. % for simplicity.

%\section{Normalization and Additivity}

\begin{lem}\label{dual}
Let $\sigma=\sigma((0,1),(\alpha,\beta))$  be a lattice cone where $\alpha, \beta$ are relatively prime positive integers.
Then $\check{\sigma}$ the dual cone of $\sigma$ is similar to
$$
\tau = \tau((0,-1),(\alpha,\beta)).
$$
\end{lem}

%{\bf Picture A (i) needed!!!}

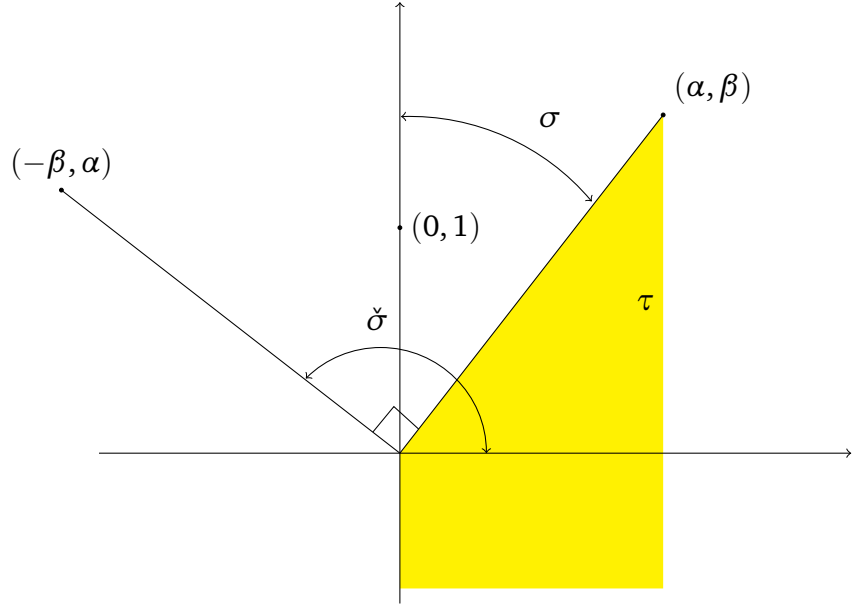
\begin{figure}
\begin{tikzpicture}
\path [fill=yellow] (0,0) -- (3.5,4.5)  -- (3.5,-1.8) -- (0,-1.8) (0,0);
\node [right] at (3,2) {$\tau$}; 
\draw [->] (-4,0) -- (6,0);
\draw [->] (0,-2) -- (0,6);
\draw [-] (0,0) -- (3.5,4.5);
\draw [-] (-0.36,0.28) -- (-0.08,0.62)--(0.25,0.32);
\draw [<->] (1.15,0) arc (0:135:40pt);
\draw [<->] (2.55,3.35) arc (40:92:90pt);
\node [above] at (2,4.2) {$\sigma$};
\node [above left] at (0,1.5) {$\check\sigma$};
\draw [-] (0,0) -- (-4.5,3.5);
\node [above] at (-4.5,3.5) {$(-\beta,\alpha)$}; 
\draw [fill] (3.5,4.5) circle [radius=0.025];
\draw [fill] (-4.5,3.5) circle [radius=0.025];
\node [above right] at (3.5,4.5) {$(\alpha,\beta)$};  
%\node [right] at (2.7,-1.6) {$B_{-1} = \iota(\omega)$};
\draw [fill] (0,3) circle [radius=0.025];
%\draw [fill] (0,3) circle [radius=0.025];
\node [right] at (0,3) {$(0,1)$};
%\draw [fill] (6.7,2.6) circle [radius=0.025];
%\draw [dotted] (0,0) -- (3,1.2);
%\node at (9.2,3.5) {$\iota(\fb^{-1})\subset K_\FR$};
\end{tikzpicture}
\caption{$\sigma$, $\check\sigma$ and $\tau$}
\label{cones}
\end{figure}

\begin{proof}
It is easy to see the dual cone $\check\sigma$ has primitive basis
$((1,0),(-\beta,\alpha))$. 
See Fig. \ref{cones}. Since the rotation by $-90$ degree belongs to $SL_2(\FZ)$,
we have the desired similarity of the cones.
%By definition of dual cone $\check{\sigma}$ we find that
%$$\check{\sigma}=\big<(1,0),(-\beta,\alpha)\big>\sim\big<(0,-1),(\alpha,\beta)\big>,$$
%because for $\begin{pmatrix}
 %  0   & 1   \\
 %   -1  & 0
%\end{pmatrix}\in GL_2(\FZ)$, we have
%$$\begin{pmatrix}
 %  0   & 1   \\
  %  -1  & 0
%\end{pmatrix}
%\begin{pmatrix}
  % 1  & -\beta   \\
  % 0  & \alpha
%\end{pmatrix}=
%\begin{pmatrix}
  % 0   & \alpha  \\
   % -1  & \beta
%\end{pmatrix}.$$
\end{proof}

%Now, we consider a cone associated to  ideal $\fb$ considered in section2,
%$$
%\sigma(\fb)=\sigma((1,0),(\alpha_{s-1},\beta_{s-1})).
%$$
After Prop. \ref{sim} and Lemma \ref{dual}, for $\fb$ as before, we have
%find that the dual cone $\check\sigma(\fb)$ is similar to cone
$$
\check{\sigma} \sim \tau:=\tau((0,-1),(\alpha_{\ell-1},\beta_{\ell-1})
$$
thus 
$$S_{\check\sigma}(x_1,x_2)=S_{\tau}(x_1,x_2).$$ 
Let
$v_{-1}=(0,1), \,\,v_{0}=(1,0)$ and for $1\leq i\leq\ell -1,$
$$v_{i}=(\alpha_i,\beta_i),$$ for $\alpha_i, \beta_i$  defined
as in Eq.\eqref{albe}.
Notice that $v_i$ corresponds to $B_{-i+1}$and $v_{-1}$, $v_0$ are the two standard basis  of $M$. 
Then the decomposition of $\sigma$ yields that of $\check\sigma$:
\begin{prop}\label{todd} With above notations, let
$\sigma_0':=\sigma_0'(-v_{-1},v_0)$ and 
$$
\sigma_i:=\sigma_i(v_{i-1},v_i),
$$ for $i\geq0$.
Then we have
$$
\check\sigma\sim \tau := \tau(\nu_{-1},\nu_\ell)= \sigma_0'+\sigma_1+\sigma_2+\sigma_3+\cdots +\sigma_{\ell-1}.
$$
Thus we have
\begin{equation*}
%\begin{split}
S_{\check{\sigma}}(x_1,x_2)=%&
F(A_{\sigma'_0}^{-1}A_{\tau}(x_1,x_2)^{t}) + \sum_{i=1}^{\ell-1} (-1)^i F(A_{\sigma_i}^{-1} A_\tau (x_1,x_2)^t).
%- F(A_{\sigma_1}^{-1}A_\sigma(x_1,x_2)^{t})
%+ F(A_{\sigma_2}^{-1}A_\sigma(x_1,x_2)^{t})-  \\
%&+ F(A_{\sigma_{\ell-1}}^{-1}A_{\sigma}(x_1,x_2)^{t}).\end{split}
\end{equation*}
\end{prop}
\begin{proof}
First, one should notice that
$$
\tau = \sigma'_0 + \rho
$$
where $\rho= \rho(\nu_0,\nu_{\ell-1})$% denotes the 1st quadrant
(See Fig. \ref{cones}).
Note also 
$$
\det(A_{\sigma_{i}}) = \det\begin{pmatrix}
   \alpha_{i-1}   & \alpha_i   \\
    \beta_{i-1}  & \beta_i
\end{pmatrix}=\beta_{i}\alpha_{i-1}-\alpha_i\beta_{i-1}=(-1)^{i-1}.
 $$
%and 
%\begin{align*}
%\alpha_{i+2}=a_{\ell-i-2}\alpha_{i+1}+\alpha_i \\
%\beta_{i+2}=a_{\ell-i-2}\beta_{i+1}+\beta_i
%\end{align*}
%for each $i$.
%Thus we have
%$$\frac{\beta_i}{\alpha_i}-\frac{\beta_{i+2}}{\alpha_{i+2}}=\frac{a_{\ell-i-2}(\beta_{i}\alpha_{i+1}-\alpha_i\beta_{i+1})}{\alpha_i\alpha_{i+2}}=\frac{a_{\ell-i-2}(-1)^{i+1}}{\alpha_i\alpha_{i+2}}.$$
%It follows that 
%$$\frac{\beta_1}{\alpha_1}>\frac{\beta_3}{\alpha_3}>\cdots>\frac{\beta_{\ell-1}}{\alpha_{\ell-1}}>
%\frac{\beta_{\ell-2}}{\alpha_{\ell-2}}>\frac{\beta_{\ell-4}}{\alpha_{\ell-4}}>\cdots>\frac{\alpha_0}{\beta_0}.$$
Hence the decomposition of $\rho$ into nonsingular cones $\sigma_i$
$$
\rho = \sigma_1 + \sigma_2 + \cdots + \sigma_{\ell-1}
$$
finishes the proof.
\end{proof}

\begin{lem}\label{mi}
For $ -1\leq i\leq \ell-1$, let
$$
M_i:=(-1)^{i+1}((\beta_i\alpha_{\ell-1}-\alpha_i\beta_{\ell-1})x_2+\alpha_ix_1).
$$
Then we have
$$
\td_{\check\sigma}(x_1,x_2)=\alpha_{\ell-1}x_1x_2 \left( 
\sum_{i=-1}^{\ell-2}(-1)^{i}
F(M_i,M_{i+1})+\frac{1}{1-e^{-\alpha_{\ell-1}x_2}}
\right),
$$
where
$F(x_1,x_2)=\frac{1}{1-e^{-x_1}}\frac{1}{1-e^{-x_2}}.$
\end{lem}

\begin{proof}
%From Prop.\ref{todd}, we find that
%\begin{equation*}
%\begin{split}
%S_{\check{\sigma}}&(x_1,x_2)=S_{\tau}(x_1,x_2)\\
%&=F(C_0'^{-1}\tau(x_1,x_2)^{t})+F(C_0^{-1}\tau(x_1,x_2)^{t})+\sum_{i=-1}^{\ell-2}(-1)^{i}F(C_{i+1}^{-1}\tau(x_1,x_2)^{t}).
%\end{split}\end{equation*}
After simple computation, we obtain that
$$
A_{\sigma_{i-1}}^{-1}A_\tau(x_1,x_2)^{t}=(M_{i+1},M_i)^{t}.$$
Since $F(-x_1,x_2)+F(x_1,x_2)=\frac{1}{1-e^{-x_2}}$, we have
\begin{equation*}
\begin{split}
&F(C_0'^{-1}\tau(x_1,x_2)^{t})+F(C_0^{-1}\tau(x_1,x_2)^{t})\\
=&F(x_1-\beta_{\ell-1}x_2,\alpha_{\ell-1}x_2)+F(-(x_1-\beta_{\ell-1}x_2),\alpha_{\ell-1}x_2)\\
=&\frac{1}{1-e^{-\alpha_{\ell-1}x_2}}
\end{split}
\end{equation*}
Note %We note that the number of $M/M((0,-1),(\alpha_{\ell-1},\beta_{\ell-1}))$ is  
$$
\det\begin{pmatrix}
     0 & \alpha_{\ell-1}   \\
     -1 &  \beta_{\ell-1}
\end{pmatrix}=\alpha_{\ell-1}.
$$
If we apply the above to Prop. \ref{todd}, we complete the proof.
\end{proof}
Let $\td_{\sigma}(x_1,x_2)^{(n)}$ be the degree $n$ homogeneous part of  $\td_{\sigma}(x_1,x_2).$
\begin{prop}\label{homo}
Let
%$g(z):=\frac{z}{1-e^{-z}}=\sum_{d=0}^{\infty}\lambda_{d}z^{d}$
%and
\begin{align*}
L_k(X,Y)&=\sum_{i=1}^{2k+1}\frac{B_{i}}{i!}\frac{B_{2k+2-i}}{(2k+2-i)!}X^{i-1}Y^{2k-i+1},\\
R_k(X,Y)&=X^{2k}+X^{2k-1}Y+ \cdots+Y^{2k}.
\end{align*}
Then we have
\begin{equation*}
\begin{split}
&\td_{\check{\sigma}}(x_1,x_2)^{(2k+2)}=\\
&\alpha_{\ell-1}\left(\sum_{i=-1}^{\ell-2}(-1)^i L_k(M_{i+1},M_i)x_1x_2+\sum_{i=1}^{\ell-1} (-1)^ia_{\ell-i}R_k(M_{i-2},M_i)x_1x_2\frac{B_{2k+2}}{(2k+2)!}\right)\\
&+\frac{B_{2k+2}}{(2k+2)!}(-x_1M_0^{2k+1}+x_2M_{\ell-2}^{2k+1})+\delta_{k,0}\frac{1}{2}\alpha_{\ell-1} x_1 x_2.
\end{split}
\end{equation*}
\end{prop}
\begin{proof}
%Let $$g(z):=\frac{z}{1-e^{-z}}=\sum_{d=0}^{\infty}(-1)^d\frac{B_{d}}{d!}z^{d}.$$
From Prop. \ref{mi}, we find that
$$
\td_{\check{\sigma}}(x_1,x_2)^{(2k+2)}=\alpha_{\ell-1}x_1x_2\sum_{i=-1}^{\ell-2}(-1)^{i}
F(M_i,M_{i+1})^{(2k)}
%\frac{g(M_i)g(M_{i+1})}{M_iM_{i+1}}^{(2k)}
+x_1\frac{\alpha_{\ell-1}x_2}{1-e^{-\alpha_{\ell-1}x_2}}^{(2k+1)}.
$$
We have
\begin{equation*}
\begin{split}
&
F(M_i,M_{i+1})^{(2k)}\\
%\frac{g(M_i)g(M_{i+1})}{M_iM_{i+1}}^{(2k)}\\
=& \sum_{m=1}^{2k+1} \frac{B_m}{m!}\frac{ B_{2k+2-m}}{(2k+2-m)!} M_{i+1}^{m-1}M_i^{2k-m+1}
+\frac{B_{2k+2}}{(2k+2)!}
\left(
\frac{M_i^{2k+1}}{M_{i+1}}+\frac{M_{i+1}^{2k+1}}{M_{i}}
\right)\\
=& L_k(M_{i+1},M_i)
+\frac{B_{2k+2}}{(2k+2)!}\left(
\frac{M_i^{2k+1}}{M_{i+1}}+\frac{M_{i+1}^{2k+1}}{M_{i}}
%M_{i+1}^{-1}M_i^{2k+1}+M_{i+1}^{2k+1}M_{i}^{-1}
\right)
\end{split}
\end{equation*}
and $$
\frac{\alpha_{\ell-1}x_2}{1-e^{-\alpha_{s-1}x_2}}^{(2k+1)}
=-\frac{B_{2k+1}}{(2k+1)!}\alpha_{\ell-1}^{2k+1}  x_2^{2k+1}
=\delta_{k,0}\frac{1}{2}\alpha_{\ell-1} x_2,
$$
as $B_{2k+1}=0$ for $k>0$.

Moreover we also have the following:
\begin{equation}\label{alter}
\begin{split}
&\sum_{i=-1}^{\ell-2}(-1)^i (M_{i+1}^{-1}M_i^{2k+1}+M_{i+1}^{2k+1}M_{i}^{-1})\\
&=-\frac{M_0^{2k+1}}{M_{-1}} +\frac{M_{\ell-2}^{2k+1}}{M_{\ell-1}}+\sum_{i=1}^{\ell-1}(-1)^{i}\frac{M_{i-2}^{2k+1}-M_{i}^{2k+1}}{M_{i-1}}.
\end{split}
\end{equation}
As
$$
\begin{cases} (\alpha_{-1},\beta_{-1})=(0,1)\\
\alpha_{i+1}=a_{\ell-i-1}\alpha_{i}+\alpha_{i-1} \\
\beta_{i+1}=a_{\ell-i-1}\beta_{i}+\beta_{i-1}
\end{cases}
$$
we have
$$
M_{-1}=\beta_{-1}\alpha_{\ell-1}x_2-\alpha_{-1}(-x_1+\beta_{\ell-1}x_2)=\alpha_{\ell-1}x_2,
$$ 
$$
M_{\ell-1}=\beta_{\ell-1}\alpha_{\ell-1}x_2-\alpha_{\ell-1}(-x_1+\beta_{\ell-1}x_2)=\alpha_{\ell-1}x_1.
$$
and
$$M_{i+1}=-a_{\ell-i-1}M_i+M_{i-1}.$$
Therefore \eqref{alter} is equal to
\begin{equation}
\begin{split}
&-\frac{M_0^{2k+1}}{\alpha_{\ell-1}x_2} +\frac{M_{\ell-2}^{2k+1}}{\alpha_{\ell-1}x_1}+\sum_{i=1}^{\ell-1}(-1)^{i}a_{\ell-i}\frac{M_{i-2}^{2k+1}-M_{i}^{2k+1}}{M_{i-2}-M_{i}}\\
&=-\frac{M_0^{2k+1}}{\alpha_{\ell-1}x_2} +\frac{M_{\ell-2}^{2k+1}}{\alpha_{\ell-1}x_1}+\sum_{i=1}^{\ell-1}(-1)^{i}a_{\ell-i}R_k(M_{i-2},M_i)
\end{split}
\end{equation}
Thus we finally complete proof.
\end{proof}

\section{Special values of zeta function}\label{special_zeta_values}
Now we are going to evaluate the values of $\zeta(s,\fb)$ at non-positive integers using the %computation
expression of the degree $n$ homogeneous part of the Todd series
made in the previous section.
We suppose $\fb$ is an integral ideal normalized as in the previous section so that $\fb^{-1}=[1,\omega]$ for 
$$
\omega=[[a_0,a_1,\cdots,a_{r-1}]].
$$
%We recall our conditions;
%Let $\fb$ be an ideal of real quadratic fields $K$ such that $\fb^{-1}=[1,\omega]$ and
%$$\omega=[[a_0,a_1,\cdots,a_{r-1}]]$$
and $\epsilon>1$ denotes the totally positive fundamental unit of $K$.
Let $\ell$ be the even period of continued fraction expansion of $\omega$.

$(\alpha_i,\beta_i)$ for $i=1,2,\cdots,\ell-1,$
%
%We define for $i=1,2,\cdots,s-1,$
%\begin{equation*}
%\left(\begin{array}{c}\alpha_{i} \\\beta_i\end{array}\right) := \left(\begin{array}{cc}a_{s-1} & 1 \\1 & 0\end{array}\right)\left(\begin{array}{cc}a_{s-2} & 1 \\1 & 0\end{array}\right).....\left(\begin{array}{cc}a_{s-i} & 1 \\1 & 0\end{array}\right)\left(\begin{array}{c}1 \\0\end{array}\right),
%\end{equation*}
%and
and 
$$(\alpha_{-2},\beta_{-2})=(1,-a_0),\,\,(\alpha_{-1},\beta_{-1})=(0,1),\,\,(\alpha_0,\beta_0)=(1,0)$$
are primitive lattice vectors in $M$. Note that $(\alpha_i,\beta_i)$ corresponds to $B_{-i+1}$
in $K_\FR$.
$$Q(x_1,x_2):=N(\fb)(x_1\omega+x_2)(x_1\omega'+x_2).$$

Then the partial zeta function $\zeta(s,\fb)$ is expressed as (See Prop.\ref{zeta}.)
$$\zeta(s,\fb)=\sum_{l\in M}\frac{wt_{\sigma(\fb^{-1})}^2(l)}{Q(l)^s}.$$
where  
$$
\sigma(\fb^{-1})=\sigma((0,1),(\alpha_{\ell-1},\beta_{\ell-1})).
$$

In Thm.\ref{Garou-Pommer}., the partial zeta value is written using the Todd differential operator of the cone 
$\check\sigma$ dual to $\sigma = \sigma(\fb^{-1})$. 
We apply the additivity of the Todd series(Prop.\ref{homo}.) after the cone decomposition of $\check{\sigma}$ 
occurring in 
the continued fraction of $\omega$ to this expression.
% to 
Then we obtain the following expression of the partial zeta value:
% using the cone decomposition occurring in 
%the continued fraction of $\omega$:

%can express zeta values using todd power series
%\begin{thm}[Theorem 2, Garoufalidis, Pommersheim]
%For $n\geq0,$ we have
\begin{equation}\label{zeta16}
\zeta(-k,\fb)= (-1)^k k! ( \mathcal{L} + \mathcal{R})\circ \int_{\sigma(h)}e^{-Q(x_1,x_2)}dx_1dx_2 \Big|_{h=0}
\end{equation}
where 
%$\mathcal{L}$ and $\mathcal{R}$ are the following differential operators:
\begin{equation}
\begin{split}
&\mathcal{L} :=
\sum_{i=-1}^{\ell-2}(-1)^i L_k(M_{i+1},M_i)(\partial_{h_1},\partial_{h_2}) \alpha_{\ell-1} \partial_{h_1}\partial_{h_2}\\
&+ \frac{B_{2k+2}}{(2k+2)!}\sum_{i=0}^{\ell-1}(-1)^i a_{\ell-i} R_k (M_{i-2},M_i)(\partial_{h_1},\partial_{h_2})\alpha_{\ell-1} \partial_{h_1}\partial_{h_2}
\end{split}\end{equation}
and 
\begin{equation}
\begin{split}
\mathcal{R}:= 
\frac{B_{2k+2}}{(2k+2)!}&\big(-a_{\ell}R_k(M_{-2},M_0)(\partial_{h_1},\partial_{h_2})\alpha_{\ell-1}\partial_{h_1}\partial_{h_2}-\partial_{h_1} M_0^{2k+1}(\partial_{h_1},\partial_{h_2})
\\ &+ \partial_{h_2} M_{\ell-2}^{2k+1}(\partial_{h_1},\partial_{h_2})\big).
\end{split}
\end{equation}
%\end{thm}
%\begin{proof}
%See Theorem 2 in \cite{Pom}
%\end{proof}
In (\ref{zeta16}), as the differential operators are linear, this expression 
can be evaluated one by one. 
Later in Sec.\ref{vanishing_section} we will prove that the part of (\ref{zeta16}) involving $\mathcal{R}$
vanishes:
\begin{equation}\label{vanishing_part}
\mathcal{R}\circ \int_{\sigma(h)}e^{-Q(x_1,x_2)}dx_1dx_2 \Big|_{h=0}=0
\end{equation}

Then it remains only to evaluate the part involving $\mathcal{L}$. 
First, we need to rewrite the integral in another coordinate
$(y_1,y_2)$ such that $(x_1,x_2) = (\alpha_{\ell-1} y_2,\beta_{\ell-1}y_2+y_1)$.
So we have $\sigma(h)$ in the new coordinate:
\begin{equation*}
\begin{split}
\sigma(h) &= \sigma(h_1,h_2)\\
&= 
\{y_1v_1+y_2v_2| (y_1v_1+y_2v_2,u_1)\geq-h_1, (y_1v_1+y_2v_2,u_2)\geq-h_2\}\\
&=\{(\alpha_{\ell-1}y_2,\beta_{\ell-1}y_2+y_1)| x_1\geq-\frac{h_1}{\alpha_{\ell-1}}, x_2\geq-\frac{h_2}{\alpha_{\ell-1}}\}.
\end{split}
\end{equation*}

In the new coordinate $(y_1,y_2)$ the integral becomes
\begin{equation}\label{basic}
\begin{split}
&\int_{\sigma(h)}e^{-Q(x_1,x_2)}dx_1 dx_2=\alpha_{\ell-1}\int_{-\frac{h_2}{\alpha_{\ell-1}}}^{\infty}\int_{-\frac{h_1}{\alpha_{\ell-1}}}^{\infty}e^{-Q(\alpha_{\ell-1}y_2,\beta_{\ell-1}y_2+y_1)}dy_1 dy_2\\
&=\alpha_{\ell-1}\int_{-\frac{h_2}{\alpha_{\ell-1}}}^{\infty}\int_{-\frac{h_1}{\alpha_{\ell-1}}}^{\infty}e^{-N(\fb)N(\epsilon y_2+y_1)}dy_1 dy_2.
\end{split}
\end{equation}

This integral applied by $\alpha_{\ell-1}\partial_{h_1}\partial_{h_2}$ is 
\begin{equation}
\label{computation1}
%$$
\alpha_{\ell-1}\partial_{h_1}\partial_{h_2}\int_{\sigma(h)}e^{-Q(x_1,x_2)}dx_1 dx_2
=e^{-N(\fb)N(\frac{h_2}{\alpha_{\ell-1}}\epsilon+\frac{h_1}{\alpha_{\ell-1}})}
%$$
\end{equation}

The above simplifies (\ref{zeta16}) quite much assuming the vanishing of (\ref{vanishing_part}):
\begin{equation}\label{zeta21}
\begin{split}
&\zeta(-k,\fb) =
 (-1)^k k! 
\Big(\sum_{i=-1}^{\ell-2} L_k(M_{i+1},M_i)(\partial_{h_1},\partial_{h_2})
\\ &+\frac{B_{2k+2}}{(2k+2)!}\sum_{i=0}^{\ell-1} (-1)^i a_{\ell-i} R_k(M_{i-2},M_i)(\partial_{h_1},\partial_{h_2})
\Big)\circ 
e^{-N(\fb)N(\frac{h_2}{\alpha_{\ell-1}}\epsilon+\frac{h_1}{\alpha_{\ell-1}})}\Big|_{h=0}
\end{split}
\end{equation}

%\begin{lem}\label{computation1}
%Let $\sigma=\sigma(v_1,v_2)$ for $v_1=(0,1), v_2=(\alpha_{s-1},\beta_{s-1})$. Then we have
%$$\alpha_{s-1}\partial_{h_1}\partial_{h_2}\int_{\sigma(h_1,h_2)}e^{-Q(x_1,x_2)}dx_1 dx_2=e^{-N(\fb)N(\frac{h_2}{\alpha_{s-1}}\epsilon+\frac{h_1}{\alpha_{s-1}})}$$
%\end{lem}
%\begin{proof}
%We recall definition of $$\sigma(h)=\{x\in \FR^2\,\,|\,\,\big<x,u_i\big>\geq -h_i \}$$
%for inward nomal vectors $u_1=(-\beta_{s-1},\alpha_{s-1})$ and $u_2=(1,0)$. Thus we have
%$u_1 = (-\beta_{s-1},\alpha_{s-1})$ and $u_2 = (1,0)$ be the inward normal vectors to $\sigma$. 
%Thus we have
%\begin{equation*}
%\begin{split}
%&\sigma(h_1,h_2)=\{y_1v_1+y_2v_2| (y_1v_1+y_2v_2,u_1)\geq-h_1, (y_1v_1+y_2v_2,u_2)\geq-h_2\}\\
%&=\{(\alpha_{s-1}y_2,\beta_{s-1}y_2+y_1)| x_1\geq-\frac{h_1}{\alpha_{s-1}}, x_2\geq-\frac{h_2}{\alpha_{s-1}}\}.
%\end{split}
%\end{equation*}
%Since $\epsilon=\alpha_{s-1}\omega+\beta_{s-1},$  
%Changing the variables $(x_1,x_2)=(\alpha_{s-1}y_2,\beta_{s-1}y_2+y_1)$, we have
%\begin{equation*}
%\begin{split}
%&\int_{\sigma(h_1,h_2)}e^{-Q(x_1,x_2)}dx_1 dx_2=\alpha_{s-1}\int_{-\frac{h_2}{\alpha_{s-1}}}^{\infty}\int_{-\frac{h_1}{\alpha_{s-1}}}^{\infty}e^{-Q(\alpha_{s-1}y_2,\beta_{s-1}y_2+y_1)}dy_1 dy_2\\
%&=\alpha_{s-1}\int_{-\frac{h_2}{\alpha_{s-1}}}^{\infty}\int_{-\frac{h_1}{\alpha_{s-1}}}^{\infty}e^{-N(\fb)N(\epsilon y_2+y_1)}dy_1 dy_2.
%\end{split}
%\end{equation*}
%Finally applying $\partial_{h_1}\partial_{h_2}$, we obtain the wanted identity.
%\end{proof}

\begin{lem}\label{computation2}
Let $A_i=\alpha_i\omega+\beta_i.$ For $-1\leq m, l\leq \ell-1$, we have
$$M_l(\partial_{h_1}, \partial_{h_2})^iM_m(\partial_{h_1}, \partial_{h_2})^j e^{-N(\fb)N(\frac{h_2}{\alpha_{\ell-1}}\epsilon+\frac{h_1}{\alpha_{\ell-1}})}|_{h=0}=$$
$$
\partial_{h_1}^i\partial_{h_2}^j e^{-N(\fb)N((-1)^{l+1}A_lh_1+(-1)^{m+1}A_mh_2)}|_{h=0},
$$
for 
$M_i(x_1,x_2)=(-1)^{i+1}\left((\beta_i\alpha_{\ell-1}-\alpha_i\beta_{\ell-1})x_2+\alpha_ix_1\right)$.

%$$
%(\frac{\partial}{\partial h_1})^i(\frac{\partial}{\partial h_2})^j e^{-Q((-1)^{l+1}\alpha_l h_1+(-1)^{m+1}\alpha_m h_2,(-1)^{l+1}\beta_l h_1+(-1)^{m+1}\beta_m h_2)}
%$$
\end{lem}
\begin{proof}
%We recall that
For simplicity, let $c_i=(-1)^{i+1}(\beta_i\alpha_{\ell-1}-\alpha_i\beta_{\ell-1})$ and $d_i=(-1)^{i+1}\alpha_i.$
By internal change of coordinate $(h_1,h_2) \mapsto (ah_1+c h_2 , b h_1+d h_2)$,
%In \cite{Pom}, 
we have
$$(a\partial_{h_1}+b\partial_{h_2})^i(c\partial_{h_1}+d\partial_{h_2})^j f(h_1,h_2)|_{h=0}=\partial_{h_1}^i\partial_{h_2}^jf(ah_1+ch_2,bh_1+dh_2)|_{h=0}.$$
Thus
\begin{equation*}
\begin{split}
M_l(\partial_{h_1},\partial_{h_2})^i & M_m(\partial_{h_1},\partial_{h_2})^j\circ e^{-N(\fb)N(\frac{h_2}{\alpha_{\ell-1}}\epsilon+\frac{h_1}{\alpha_{\ell-1}})}|_{h=0}\\
=&(d_l\partial_{h_1}+c_l \partial_{h_2})^i(d_m\partial_{h_1}+c_m \partial_{h_2})^j\circ e^{-N(\fb)N(\frac{h_2}{\alpha_{\ell-1}}\epsilon+\frac{h_1}{\alpha_{\ell-1}})}|_{h=0}\\
=&\partial_{h_1}^i\partial_{h_2}^j\circ e^{-N(\fb)N(\frac{d_lh_1+d_mh_2}{\alpha_{\ell-1}}\epsilon+\frac{c_lh_1+c_mh_2}{\alpha_{\ell-1}})}|_{h=0}
\end{split}
\end{equation*}

We note that
$$
\beta_{\ell-1}d_i+c_i=(-1)^{i+1}\beta_i\alpha_{\ell-1}.
$$
Since $\epsilon=\alpha_{\ell-1}\omega+\beta_{\ell-1}$,
we have
\begin{equation*}
\begin{split}
(d_lh_1&+d_mh_2)\epsilon+c_lh_1+c_mh_2\\
%=& (d_lh_1+d_mh_2)\beta_{s-1}+c_lh_1+c_mh_2+(d_lh_1+d_mh_2)\alpha_{s-1}\omega\\
=& (d_l\beta_{\ell-1}+c_l )h_1+(d_m\beta_{\ell-1}+c_m )h_2+(d_lh_1+d_mh_2)\alpha_{\ell-1}\omega\\
=& \alpha_{\ell-1}\Big{(}((-1)^{l+1}\alpha_lh_1+(-1)^{m+1}\alpha_mh_2)\omega+(-1)^{l+1}\beta_lh_1+(-1)^{m+1}\beta_mh_2\Big{)}\\
=& \alpha_{\ell-1}((-1)^{l+1}A_l h_1 + (-1)^{m+1} A_m h_2)
\end{split}
\end{equation*}
\end{proof}

We note that 
\begin{equation*}
\begin{split}
&N(\fb)N((-1)^{l+1}A_lh_1+(-1)^{m+1}A_mh_2)\\
&=Q((-1)^{l+1}\alpha_lh_1+(-1)^{m+1}\alpha_m h_2,(-1)^{l+1}\beta_lh_1+(-1)^{m+1}\beta_m h_2),
\end{split}
\end{equation*}
for a binary quadratic form $Q(x,y)$ with degree $2$ and $i ,j$ with $i+j=2k$, 
we have
$$\partial_{h_1}^i\partial_{h_2}^j e^{-Q(h_1,h_2)}\Big|_{h=0}=(-1)^k\frac{1}{k!}\partial h_1^i\partial h_2^j Q(h_1,h_2)^k\Big|_{h=0}$$
Thus, from (\ref{computation1}) (\ref{zeta21}) and Lemma \ref{computation2}, we have 
finished the proof of our second main theorem(Thm. \ref{2nd_main}).
\begin{equation*}
\begin{split}
%&(-1)^n n!\Big{\{}
%\sum_{i=-1}^{s-2}(-1)^i L_n(M_{i+1},M_i) \alpha_{s-1} \partial_{h_1}\partial_{h_2}
%+ \frac{B_{2n+2}}{(2n+2)!}\sum_{i=1}^{s-1}(-1)^i a_{s-i} R_n (M_{i-2},M_i)\alpha_{s-1} \partial_{h_1}\partial_{h_2}\Big{\}}\\
%&\circ\int_{\sigma(h_1,h_2)}e^{-Q(x_1,x_2)}dx_1dx_2
&\zeta(-k,\fb) =\\
&\sum_{i=0}^{\ell-1}(-1)^{i-1}L_k(\partial_{h_1},\partial_{h_2})Q(\alpha_ih_1-\alpha_{i-1}h_2,\beta_i h_1-\beta_{i-1}h_2)^k\Big{|}_{h=0}\\
&+\frac{B_{2k+2}}{(2k+2)!}\sum_{i=0}^{\ell-1}(-1)^ia_{\ell-i}R_n(\partial_{h_1},\partial_{h_2})Q(\alpha_{i-2}h_1+\alpha_{i}h_2,\beta_{i-2}h_1+\beta_{i}h_2)^k\Big{|}_{h=0}.
\end{split}
\end{equation*}

\begin{remark}
We have  obtained a polynomial expression of the zeta value in variables $\alpha_i, \beta_i$ and
the coefficients of the quadratic form $Q(x,y)$. For the polynomial expression, it is important to show the vanishing \eqref{vanishing_part}. In Sec. \ref{vanishing_section}, there appear $\alpha_0, \alpha_{\ell-1}$ in the denominator of the vanishing expression involving  the $\mathcal{R}$-operator.
This is a crucial ingredient of the Kummer congruence and  
the corresponding p-adic zeta function.
\end{remark}

\section{Computation of $\zeta(-k,\fb)$ for  $k=0,1$ and $2$}\label{computations}

In this section, we evaluate the zeta values $\zeta(-k,\fb)$ explicitly for small $n$.
We express the values in terms of the continued fraction expansion 
$[[a_0, \a_1, \ldots, a_{\ell-1}]]$ of the reduced basis $\omega$ of $\fb^{-1}$.

\subsection{k=0}
For $k=0$, the zeta value is already known by C. Meyer(\cite{Me}) in terms of negative continued fraction.
Using the plus-to-minus conversion formula of continued fraction
\begin{equation}\label{ptom}
\begin{split}
\delta & = \omega +1 = [[ a_0 , a_1, \ldots, a_{\ell-1}]] +1\\
 &= ((a_0+2,2,\ldots,2, a_2+2, 2,\ldots,2, a_4+2,\ldots, a_{\ell-2}+2, 2,\ldots,2))\\
 &= ((b_0,b_1,\ldots,b_{m-1}))
\end{split}\end{equation}
one obtains the result in positive continued fraction.

In our approach, we begin with the expression using positive continued fraction as a special case of Thm.\ref{2nd_main}:
$$
\zeta(0,\fb)=
\frac{B_2}{2}\sum_{i=0}^{\ell-1}(-1)^i a_{\ell-i}
$$
Since $B_2=1/6$ and $\ell$ is the even period of the continued fraction,
this reduces to 
\begin{equation}
\zeta(0,\fb)=
\frac{1}{12} \sum_{i=0}^{\ell-1}(-1)^i a_i
\end{equation}
Via \eqref{ptom}, one recovers the result of Meyer:
$$
\zeta(0,\fb) = \frac1{12}\sum_{i=0}^{m-1} (b_i -3),
$$
where $b_i$ is the $i$-th term of the negative continued fraction.

\begin{remark}
Note that using the positive continued fraction, we have an alternating sum for the zeta value.
Consequently, one sees directly the vanishing of $\zeta(0,\fb)$ when the actual period of the positive continued fraction
of $\omega$ is odd (equivalently, if the fundamental unit is not totally positive).
\end{remark}

\subsection{k=1 and 2}

For $Q(x_1,x_2) = N(\fb)N(x_1 \omega + x_2)$,
let $L_i$, $M_i$ and $N_i$ be defined as in
$$
Q(\alpha_i h_1 - \alpha_{i-1} h_2,\beta_i h_1  - \beta_{i-1} h_2)
= L_i h_1^2 + M_i h_1 h_2 + N_i h_2^2
$$
Similarly, 
$\tilde{L}_i$, $\tilde{M}_i$ and $\tilde{N}_i$ are defined as follows:
$$
Q(\alpha_{i-2} h_1 + \alpha_{i} h_2,\beta_{i-2} h_1 + \beta_{i} h_2)
= \tilde{L}_i h_1^2 + \tilde{M}_i h_1 h_2 + \tilde{N}_i h_2^2
$$
Then the special value at $s=-1$ is computed out as follows:
\begin{equation*}
\begin{split}
\zeta(-1,\fb)  &= \sum_{i=0}^{\ell-1} (-1)^{i-1} \frac{B_2^2}{4} M_i 
	+ \frac{B_4}{4!} \sum_{i=0}^{\ell-1} (-1)^i a_{\ell-i} (2 \tilde{L}_i + \tilde{M}_i + 2 \tilde{N}_i)\\
	& = \frac{1}{720} \sum_{i=0}^{\ell-1} (-1)^{i-1}\left(5 M_i + a_{\ell-i} (2 \tilde{L}_i +\tilde{M}_i + 2 \tilde{N}_i)\right)
\end{split}
\end{equation*}

Similarly for $s=-2$, 
\begin{equation*}
\begin{split}
\zeta & (-2,\fb) = \\
&\frac{1}{15120} \sum_{i=0}^{\ell-1} (-1)^i 
\left(21 M_i (N_i+L_i)
+ 2 a_{\ell-i} \left(6 \tilde{L}_i^2 + 3 \tilde{L}_i\tilde{M}_i + \tilde{M}_i^2
+ 2 \tilde{L}_i\tilde{N}_i + 3 \tilde{M}_i \tilde{N}_i + 6\tilde{N}_i^2\right)\right)
\end{split}
\end{equation*}

This should be compared with the expression obtained using negative continued fraction in \cite{Pom} and also \cite{zagier}.
They considered the zeta function of the following quadratic form in view of negative continued fraction:
$$
Q'(x_1,x_2) := N(\fb)N(x_1\delta + x_2)
$$
for $\delta =\omega +1$(See \eqref{ptom} for negative continued fraction).
Let $A_i$ be the lattice points of the component of the Klein polyhedron of $\fb^{-1}$ in the 1st quadrant with normalization:
$A_0 = 1$, $A_{-1}=\delta$  and the 1st coordinate of $A_i$ increasing according to $i$. 
Then we associate a lattice vector $(p_k, q_k)$ to $A_k$ for $A_k = - p_k A_{-1} + q_k A_0$. $p_k$ and $q_k$ are
obtain from the reduced fraction of the truncation after $k$ the of the negative continued fraction $\delta=((b_0,b_1,\ldots,b_m))$:
$$
\frac{q_k}{p_k} = (b_0,\ldots, b_{k-1})
$$
(This is the last line of pp.18 of \cite{Pom}, where $\frac{p_k}{q_k}$ should be corrected to $\frac{q_k}{p_k}$ as we just wrote above).
Similarly, $L'_i$,  $M'_i$, $N'_i$ and $\tilde{L_i}', \tilde{M_i}', \tilde{N_i}'$ are defined as the coefficients
of quadratic forms:
$$
Q'(-p_{i -1}h_1 - p_{i} h_2, q_{i-1} h_1  + q_{i} h_2)
= L_i' h_1^2 + M_i' h_1 h_2 + N_i' h_2^2
$$
and
%$\tilde{L}_i$, $\tilde{M}_i$ and $\tilde{N}_i$ are defined as follows:
$$
Q'(-p_{i -1}h_1 - p_{i+1} h_2, q_{i-1} h_1  + q_{i+1} h_2)
= \tilde{L}'_i h_1^2 + \tilde{M}'_i h_1 h_2 + \tilde{N}'_i h_2^2.
$$
In this setting, Garoufalidis-Pomersheim(\cite{Pom}) obtained:
\begin{equation*}
\zeta(-1,\fb) = \frac{1}{720} \sum_{i=0}^{m-1} \left( 5 M'_i + b_i (-2 \tilde{L}_i + \tilde{M}_i - 2\tilde{N}_i)\right)
\end{equation*} 
and
\begin{equation*}
\begin{split}
\zeta(-2,\fb)=&\frac1{15120}\sum_{i=0}^{m-1} (-21 M'_i (L'_i + N'_i) \\
&+ 2b_i(6 \tilde{L}'^2_i - 3\tilde{L}'_i\tilde{M}'_i + 2\tilde{L}'_i\tilde{N}'_i +\tilde{M}'^2 - 3\tilde{M}'_i \tilde{N}'_i + 6 \tilde{N}'^2_i)).
\end{split}
\end{equation*}
%\subsection{n=2}
It should be also compared with Zagier's result(eg. for $k=1$) in \cite{zagier}:
$$
\zeta(-1,\fb) = \frac{1}{720} \sum_{i=0}^{m-1} \left(-2 N_i b_i^3 + 3M_i b_i^2 -  6 L_i b_i + 5 M_i )\right)
$$

\begin{remark}
If we use the formula for zeta values using negative continued fractions 
as is made by Garoufalidis-Pommersheim and Zagier,  one can still obtain polynomial behavior in a family 
similar to Sec. \ref{polynomial_family} after the uniformity of the negative continued fractions in the family.
But as is known, there is hardly a direct arithmetic property(eg. regulator) associated to the negative continued fractions.
Actually in \cite{J-L1}, \cite{J-L2}, the formula for negative continued fractions, which is developed by Yamamoto(\cite{Yama}) and Zagier(\cite{zagier}, \cite{zagier2}), is used after conversion of positive continued fraction into negative one.
The formula using positive continued fraction simplifies this unnecessary step and justifies the reason of our earlier results.
\end{remark}

\section{Vanishing Part}\label{vanishing_section}
Now, it remains to show the vanishing of (\ref{vanishing_part}) 
$$
\mathcal{R}
\circ\int_{\sigma(h)}e^{-Q(x_1,x_2)}dx_1dx_2\Big|_{h=0}=0.
$$

This part is crucial in expressing the zeta values at nonpositive integers as polynomials of its argument coming from terms of continued fractions and the coefficients of quadratic forms. The vanishing has been
already observed in related works by Zagier(\cite{zagier}) and Garoufalidis-Pommersheim(\cite{Pom})
in different settings. In \cite{Pom}, only the vanishing is mentioned without clear proof. In this section, we 
will recycle some notions and ideas from \cite{zagier}. 

It suffices to show the vanishing of the following, which equals the above up to multiplication by a constant.
%This equals 
\begin{equation}\begin{split}\label{22}
\Big(-a_{\ell}R_k(M_{-2},M_0)(\partial_{h_1},\partial_{h_2})\alpha_{\ell-1}\partial_{h_1}\partial_{h_2}-\partial_{h_1}M_0^{2k+1}(\partial_{h_1},\partial_{h_2})\\
+\partial_{h_2}M_{\ell-2}^{2n+1}(\partial_{h_1},\partial_{h_2})\Big)\circ 
\int_{\sigma_{(h_1,h_2)}}e^{-Q(x_1,x_2)}dx_1dx_2\Big|_{h=0}
\end{split}
\end{equation}
%up to constant factor. 
%So it suffices to show the vanish of the above.

As $M_0=\beta_{\ell-1}x_2-x_1$ and $M_{\ell-2}=x_2-\alpha_{\ell-2}x_1$, we have
\begin{equation}\label{23}
\begin{split}
&x_1M_0^{2k+1}+x_2M_{\ell-2}^{2k+1}\\
&=2 x_1^{2k+2}+\sum_{i=1}^{2k+1}
(-1)^i \left(\begin{array}{c}2k+1 \\i\end{array}\right) (\beta_{\ell-1}^i+\alpha_{\ell-2}^i)x_2^ix_1^{2k+2-i}.
\end{split}
\end{equation}

Applying $\partial_{h_1}^{2k+2}$ to \eqref{basic},
we obtain
\begin{equation}\label{24}
\begin{split}
&\alpha_{\ell-1}\partial_{h_1}^{2k+2}\int_{-\frac{h_2}{\alpha_{\ell-1}}}^{\infty}\int_{-\frac{h_1}{\alpha_{\ell-1}}}^{\infty}e^{-N(\fb)N(\epsilon y_2+y_1)}dy_1dy_2\Big{|}_{h=0}\\
&=\frac{1}{\alpha_{\ell-1}}\int_{0}^{\infty}\partial_{h_1}^{2k+1}e^{-N(\fb)N(\epsilon \frac{y_2}{\alpha_{\ell-1}}-\frac{h_1}{\alpha_{\ell-1}})}\Big{|}_{h_1=0} dy_2
\end{split}
\end{equation}

If we write $P(x_1,x_2)=\frac{N(\fb)}{\alpha_{\ell-1}^2}
(x_2^2+(\epsilon+\epsilon')x_1x_2+x_1^2)$, using \eqref{23}-\eqref{24}, one can simplify the 2nd half of \eqref{22}:
\begin{equation}\label{25}
\begin{split}
\Big( -\partial_{h_1}&M_0(\partial_{h_1},\partial_{h_2})^{2k+1}+\partial_{h_2}M_{\ell-2}(\partial_{h_1},\partial_{h_2})^{2k+1} \Big) 
\circ \int_{\sigma_{(h_1,h_2)}}e^{-Q(x_1,x_2)}dx_1dx_2\Big|_{h=0}\\
=&\frac{1}{\alpha_{\ell-1}}\sum_{i=1}^{2k+1}(-1)^i
\begin{pmatrix}2k+1 \\i
\end{pmatrix}
(\beta_{\ell-1}^i+\alpha_{\ell-2}^i)\partial_{h_2}^{i-1}\partial_{h_1}^{2k+1-i}\circ e^{-P(h_1,h_2)}\Big|_{h=0}\\
&+\frac{2}{\alpha_{\ell-1}}\int_{0}^{\infty}\partial_{x_1}^{2k+1} e^{-P(-x_1,x_2)}\Big{|}_{x_1=0} dx_2
\end{split}
\end{equation}

From
$M_{-2} =(a_0\alpha_{\ell-1}+\beta_{\ell-1})x_2-x_1$ and
$M_0 =\beta_{\ell-1}x_2-x_1$,
we have
\begin{equation}\label{26}
R_k(M_{-2},M_0) =\sum_{i=0}^{2k+1}(-1)^{i+1}
\begin{pmatrix}
2k+1 \\i\end{pmatrix}
\frac{(a_0\alpha_{\ell-1}+\beta_{\ell-1})^i-\beta_{\ell-1}^i}{a_0\alpha_{\ell-1}}
x_2^{i-1}x_1^{2k+1-i}.
\end{equation}

From (\ref{25}) and  (\ref{26}), we obtain the following lemma:
\begin{lem} \label{vani1}Let $P(x_1,x_2)=\frac{N(\fb)}{\alpha_{\ell-1}^2}
(x_2^2+(\epsilon+\epsilon')x_1x_2+x_1^2).
$ 
Then we have
\begin{equation*}
\begin{split}
\Big{(}  -a_{\ell}R_k(M_{-2},M_0)&(\partial_{h_1},\partial_{h_2})\alpha_{\ell-1}\partial_{h_1}\partial_{h_2}-\partial_{h_1}M_0^{2k+1}(\partial_{h_1},\partial_{h_2})\\
&+\partial_{h_2}M_{\ell-2}^{2n+1}(\partial_{h_1},\partial_{h_2})\Big{)}\circ \int_{\sigma_{(h)}}e^{-Q(x_1,x_2)}dx_1dx_2\Big|_{h=0}\\
=  \frac{1}{\alpha_{\ell-1}}\sum_{i=0}^{2k}(-1)^{i+1} &
\begin{pmatrix}
2k+1 \\i+1
\end{pmatrix}  \big((a_0\alpha_{\ell-1}+\beta_{\ell-1})^{i+1}+(\alpha_{\ell-2})^{i+1}\big)\partial_{h_2}^{i}\partial_{h_1}^{2k-i}\circ e^{-P(h)}\Big|_{h=0}\\
&+\frac{2}{\alpha_{\ell-1}}\int_{0}^{\infty}\partial_{x_1}^{2k+1} e^{-P(-x_1,x_2)}\Big{|}_{x_1=0} dx_2.
\end{split}
\end{equation*}
\end{lem}

\begin{lem} \label{vani2} For the totally positive fundamental unit $\epsilon>1,$ we have 
$$\epsilon+\epsilon'=a_0\alpha_{\ell-1}+\beta_{\ell-1}+\alpha_{\ell-2}.$$
\end{lem}
\begin{proof}
We note that
$$\delta:=-\frac{1}{\omega'}=[[a_{\ell-1},a_{\ell-2},\cdots,a_0]].$$
Thus
$$
\delta= \left(\begin{array}{cc}a_{\ell-1} & 1 \\1 & 0\end{array}\right)\left(\begin{array}{cc}a_{\ell-2} & 1 \\1 & 0\end{array}\right).....\left(\begin{array}{cc}a_{0} & 1 \\1 & 0\end{array}\right)\left(\begin{array}{c}\delta \\1\end{array}\right)=\frac{\alpha_{\ell}\delta+\alpha_{\ell-1}}{\beta_{\ell}\delta+\beta_{\ell-1}}.
$$

And we have
$$\alpha_{\ell-1}\omega^2-(\alpha_{\ell}-\beta_{\ell-1})\omega-\beta_{\ell}=0.$$
Finally we have
$$\omega+\omega'=\frac{\alpha_{\ell}-\beta_{\ell-1}}{\alpha_{\ell-1}}.$$
$$\epsilon+\epsilon'=\alpha_{\ell-1}(\omega+\omega')+2\beta_{\ell-1}=\alpha_{\ell}+\beta_{\ell-1}$$
Thus
$$\frac{\epsilon+\epsilon'-\beta_{\ell-1}-\alpha_{\ell-2}}{\alpha_{\ell-1}}=a_0.$$
\end{proof}
From Lemma \ref{vani2}, if we  let $a_0\alpha_{\ell-1}+\beta_{\ell-1}=-a,$ $\alpha_{\ell-2}=-b$ and $\frac{N(\fb)}{\alpha_{\ell-1}^2}=A$ then we find that
$\epsilon+\epsilon'=-(a+b).$
Hence one can rewrite Lemma \ref{vani1} as follows:
 \begin{equation}\label{27}
\begin{split}
\sum_{i=0}^{2k}(-1)^{i+1}
\begin{pmatrix}
2k+1 \\i+1
\end{pmatrix} & \big((a_0\alpha_{\ell-1}+\beta_{\ell-1})^{i+1}+(\alpha_{\ell-2})^{i+1}\big)\partial_{h_2}^{i}\partial_{h_1}^{2k-i}\circ e^{-P(h_1,h_2)}\Big{|}_{h=0}\\
&+2\int_{0}^{\infty}\partial_{x_1}^{2k+1} e^{-P(-x_1,x_2)}\Big{|}_{x_1=0} dx_2\\
=\sum_{i=0}^{2k}
\begin{pmatrix}
2k+1 \\i+1\end{pmatrix}&
(a^{i+1}+b^{i+1})\partial_{h_2}^{i}\partial_{h_1}^{2k-i}\circ e^{-A(h_2^2-(a+b)h_1h_2+h_1^2)}\Big{|}_{h=0}\\
&+2\int_{0}^{\infty}\partial_{x_1}^{2k+1} e^{-A(x_2^2+(a+b)x_1x_2+x_1^2)}\Big{|}_{x_1=0} dx_2.
\end{split}
\end{equation}

Hence it remains to show vanishing of the right hand side of (\ref{27}):
\begin{equation}\tag{$\star$}\label{star}
\begin{split}
\sum_{i=0}^{2k} &
\begin{pmatrix}
2k+1 \\i+1\end{pmatrix}
(a^{i+1}+b^{i+1})\partial_{h_2}^{i}\partial_{h_1}^{2k-i}\circ e^{-A(h_2^2-(a+b)h_1h_2+h_1^2)}\Big{|}_{h=0}\\
&+2\int_{0}^{\infty}\partial_{x_1}^{2k+1} e^{-A(x_2^2+(a+b)x_1x_2+x_1^2)}\Big{|}_{x_1=0} dx_2.
\end{split}
\end{equation}

For the proof, we introduce $f_k(\alpha,\beta,\gamma)$ and $d_{r,k}(\alpha,\beta,\gamma)$ as follows:
\begin{align}\label{9.7}
\int_{0}^{\infty}\partial_{x_1}^{2k+1} e^{-(\alpha x_2^2+\beta x_1x_2+\gamma x_1^2)}\Big{|}_{x_1=0} dx_2 &=-\frac{(2k+1)! f_k(\alpha,\beta,\gamma)}{2\gamma^{k+1}}\\
\label{9.8}
\sum_{i=0}^{2k}d_{i,2k-i}(\alpha,\beta,\gamma)x_1^ix_2^{2k-i}& =(\alpha x_1^2+\beta x_1x_2+\gamma x_2^2)^k
\end{align}
These numbers are originally appeared in \cite{zagier}. 
One should  see that $f_k(\alpha,\beta,\gamma)$ is odd function w.r.t. $\beta$:
$$
f_k(\alpha,\beta,\gamma) + f_k(\alpha,-\beta,\gamma) = 0.
$$
One can identify $d_{i,2k-i}(\alpha,-\beta,\gamma)$ in the following expression:
\begin{equation*}
\begin{split}
&\partial_{x_1}^i\partial_{x_2}^{2k-i} e^{-(\alpha x_1^2-\beta x_1x_2+\gamma x_2^2)}\Big{|}_{(x_1,x_2)=(0,0)}\\
&=\frac{(-1)^k}{k!}\partial_{x_1}^i\partial_{x_2}^{2k-i}(\alpha x_1^2-\beta x_1x_2+\gamma x_2^2)^k=\frac{(-1)^k}{k!}i!(2k-i)!d_{i,2k-i}(\alpha,-\beta,\gamma).
\end{split}
\end{equation*}
From this,
% and (\ref{9.7})-(\ref{9.8}), 
one can rewrite the 1st line of (\ref{star}) as
\begin{equation}\label{eq1}
\begin{split}
&\sum_{i=0}^{2k}
%\left(
\begin{pmatrix}
%array}{c}
2k+1 \\i+1\end{pmatrix}
%array}\right) 
(a^{i+1}+b^{i+1})\partial_{h_2}^{i}\partial_{h_1}^{2k-i}\circ e^{-A(h_2^2-(a+b)h_1h_2+h_1^2)}
\Big|_{h=0}
\\
&=\frac{(-1)^k}{k!}(2k+1)!\sum_{i=0}^{2k}\frac{a^{i+1}+b^{i+1}}{i+1}d_{i,2k-i}(A,-A(a+b),A)
\end{split}
\end{equation}
The 2nd line of (\ref{star}) is, from the definition of $f_k(\alpha,\beta,\gamma)$,
\begin{equation}\label{eq2}
%\begin{split}
2\int_{0}^{\infty} \partial_{x_1}^{2k+1} e^{-A(x_2^2+(a+b)x_1x_2+x_1^2)}\Big{|}_{x_1=0} dx_2
=-(2k+1)! \frac{f_k(A,A(a+b),A)}{A^{k+1}}.
%\end{split}
\end{equation}

Now, we are going to use an identity relating $f_k(\alpha,\beta,\gamma)$ and $d_{i,2k-i}(\alpha,-\beta,\gamma)$
due to Zagier:
%
%To prove above proposition, we recall Zagier's formula for a relation between $f_k(\alpha,\beta,\gamma)$ and $d_{i,2k-i}(\alpha,-\beta,\gamma)$
%$\int_{0}^{\infty}\partial_{x_1}^{2k+1} e^{-(\alpha x_2^2+\beta x_1x_2+\gamma x_1^2)}\Big{|}_{x_1=0} dx_2$ and $\partial_{h_2}^{i}\partial_{h_1}^{2k-i}\circ e^{-(\alpha x_2^2-\beta x_1x_2+\gamma x_1^2)}.$
\begin{lem}[Zagier(Prop. 4 of \cite{zagier})]
For a real number $\lambda$, we have
\begin{equation*}\begin{split}
f_k(\alpha,\beta,\gamma)&+f_k(\gamma,2\lambda \gamma-\beta,\lambda^2 \gamma-\lambda \beta +\alpha)\\&=2 \frac{(-1)^k}{k!}\gamma^{k+1}\sum_{i=0}^{2k}d_{i,2k-i}(\alpha,-\beta,\gamma)\frac{\lambda^{i+1}}{i+1}.\end{split}
\end{equation*}
\end{lem}
%\begin{proof}
%See Prop. 4 of \cite{zagier}.
%\end{proof}

%\begin{lem}\label{mi} 
%For real numbers $(\alpha,\beta,\gamma)$, we have 
%$$f_k(\alpha,\beta,\gamma)+f_k(\alpha,-\beta,\gamma)=0.$$
%\end{lem}

If we put 
$\alpha=A,\beta=A(a+b),\gamma=A$ and $\lambda=a$ (resp. $\lambda=b$)
into the above,  we obtain
\begin{equation*}\label{eq3}
\begin{split}
f_k(A,A(a+b),A)&+f_k(A,A(a-b), A(-ab+1))\\&=2 \frac{(-1)^k}{k!}A^{k+1}\sum_{i=0}^{2k} d_{i,2k-i}(A,-A(a+b),A)\frac{a^{i+1}}{i+1}
\end{split}
\end{equation*}
and
\begin{equation*}\begin{split}
f_k(A,A(a+b)&,A)+f_k(A,A(b-a), A(-ab+1))\\&=2 \frac{(-1)^k}{k!}A^{k+1}\sum_{i=0}^{2k} d_{i,2k-i}(A,-A(a+b),A)\frac{b^{i+1}}{i+1}.
\end{split}
\end{equation*}

%\noindent {\bf Proof of Proposition}:
%If we take $\alpha=A,\beta=A(a+b),\gamma=A$ and $\lambda=a$[resp. $\lambda=b$]then we obtain
%\begin{equation}\label{eq3}
%\begin{split}
%&f_k(A,A(a+b),A)+f_k(A,A(a-b), A(-ab+1))\\&=2 \frac{(-1)^k}{k!}A^{k+1}\sum_{i=0}^{2k} d_{i,2k-i}(A,-A(a+b),A)\frac{a^{i+1}}{i+1}.\\
%&f_k(A,A(a+b),A)+f_k(A,A(b-a), A(-ab+1))\\&=2 \frac{(-1)^k}{k!}A^{k+1}\sum_{i=0}^{2k} d_{i,2k-i}(A,-A(a+b),A)\frac{b^{i+1}}{i+1}.
%\end{split}
%\end{equation}

As $f_k$ is odd function of its 2nd argument, summing the above two equations, we have 
%
%If we sum two equations in (\ref{eq3}) then  from lemma \ref{mi}, we have
\begin{equation*}\label{eq4}
\begin{split}
\frac{f_k(A,A(a+b),A)}{A^{k+1}}
= \frac{(-1)^k}{k!}\sum_{i=0}^{2k} d_{i,2k-i}(A,-A(a+b),A)\frac{a^{i+1}+b^{i+1}}{i+1}.
\end{split}
\end{equation*}
This identifies (\ref{eq1}) and (\ref{eq2}) up to sign. % change of sign. 

Therefore we concludes the vanishing of (\ref{star}).
%From (\ref{eq1}), (\ref{eq2}) and (\ref{eq4}), we complete the proof.
%Therefore, we have finished the proof.

\section{Appllication: Polynomial behavior of zeta values at nonpositive integers in family}\label{polynomial_family}

Until now, we developed a way to compute the partial zeta values at nonpositive integers 
for a real quadratic field with a fixed ideal $\fb$ via the shape of the continued fraction
of $\omega$ for $\fb^{-1}=[1,\omega]$. We will apply this method to certain families
of real quadratic fields to prove the main theorem of this paper(Thm. \ref{polynomial_behavior}). We deal with the same family of real quadratic fields with ideals fixed
as in our earlier works (\cite{J-L1}, \cite{J-L2}).
In our previous works, the partial Hecke's $L$-values and the partial zeta values of a ray class ideal at $s=0$ are investigated for family of real quadratic fields. 
We showed that the values in the family is given by a quasi-polynomial in variable $n$ which is 
the index of the family of real quadratic fields considered. 
If the conductor is trivial, so that we consider
ideal classes, the values behave actually in a polynomial. 
This method was originally observed by Bir\'o and 
has been main ingredient to solve
class number problems of the real quadratic fields in the family without relying on 
the Riemann hypothesis(cf. \cite{Biro1}, \cite{Biro2}, \cite{Lee1}, \cite{Lee2}, \cite{B-L}).
Here we deal with the case when the conductor
is trivial. Thus we have strict polynomial instead of quasi-polynomials.

We generalize the result on the partial zeta values at $s=0$ to every nonpositive integer $s$ when the
conductor is trivial. This means we consider partial zeta function of ideal classes instead of ray classes.
So the scope of partial zeta functions we consider here is narrower than the previous. 
But the same method must be applicable to ray class partial zeta functions.
% and will produce quasi-polynomials. 
In this case by the same reason quasi-polynomials are appearing instead of polynomials to give the zeta values at a given nonpositive integer for the same family of ideals. Again this will answer the same for 
the partial Hecke L-values at arbitrary non-positive integers.

%The family we consider  is  as follows.
%We consider the ideal class partial zeta values at nonpositive integers in the family.
%Now we come to the proof of our main theorem(Thm.\ref{polynomial_behavior}). 

Recall the conditions on the family $(K_n, \fb_n)$ indexed by $n\in N$ for a subset $N$ of $\FN$. 
%should satisfy 
%the following conditions:
%
%Now, we consider the family of real quadratic fields
%$\{K_n\}_{n\in N}$ satisfying the following conditions. 
$\fb_n^{-1} = [1,\omega(n)]$ for a reduced
%A real
%quadratic field $K_n$ contains an ideal $\fb_n=[1,\omega(n)]$ for
%a reduced 
element $\omega(n)\in K_n$ and 
%Moreover, $\omega(n)$ has
%the following continued fraction expansion
$$
\omega(n)=[[a_0(n),a_1(n),\cdots,a_{r-1}(n)]]
$$
for polynomials $a_i(x)\in \FZ[x]$ and the quadratic form
$N(\fb_n)(x\omega(n)+y)(x\omega(n)'+y)$ associated with $\fb_n$ is
expressed as
$$b_1(n)x^2+b_2(n) xy +b_3(n) y^2$$ for polynomials
$b_i(x)\in \FZ[x].$ 

%\begin{thm}
%With above hypothesis,
%for a nonnegative integer $k$ fixed,
%$\zeta(-k,\fb_n)$ is, as function of $n$, equal to a polynomial in variable $n$
%on $N$. 
%The degree is given by  
%$kC+ D$ 
%and the coefficients are in $\frac{1}{C_k}\FZ$,
%where $C$ and $D$ are given as follows:
%\begin{equation*}
%\begin{split}
%C&=2 deg \alpha_{s-1}+ deg b_1 \\
%Max_{1\leq i\leq 3}\{deg A_i\} 
%D&=Max_{0\leq i\leq r-1}\{deg a_i\}
%\end{split}
%\end{equation*}
%and 
%$$
%C_k= \text{the denominator of } \sum_{0\leq i \leq k} \frac{B_{i+1}}{i+1}\frac{B_{2k+1-i}}{2k+1-i} +\frac{B_{2k+2}}{(2k+2)(2k+1) }
%\begin{pmatrix}2k \\ i\end{pmatrix} ^{-1}.
%$$
%\end{thm}

\begin{proof}[Proof of Thm.\ref{polynomial_behavior}]
Applying  Thm.\ref{2nd_main} to the family considered,  we have 
\begin{equation*}
\begin{split}
&\zeta(-k,\fb_n) =\\
&\sum_{i=0}^{\ell-1}(-1)^{i-1}L_k(\partial_{h_1},\partial_{h_2})Q(\alpha_i (n) h_1-\alpha_{i-1} (n) h_2,
\beta_i (n) h_1-\beta_{i-1}(n) h_2)^k \\  %\Big{|}_{h=(0,0)}\\
&+\frac{B_{2k+2}}{(2k+2)!}
	\sum_{i=0}^{\ell-1}(-1)^i a_{\ell-i}(n) R_k(\partial_{h_1},\partial_{h_2})
	Q(\alpha_{i-2}(n) h_1+\alpha_{i}(n) h_2,\beta_{i-2}(n) h_1+\beta_{i}(n) h_2)^k. %\Big{|}_{h=(0,0)}.
\end{split}
\end{equation*}

Since $Q(-)$ is a quadratic form and $a_i(n),b_i(n), \alpha_i(n), \beta_i(n)$ are polynomials,
it is clear that $\zeta(-k,\fb_n)$ is a polynomial in $n$.

Notice that 
%\begin{equation*}
%\begin{split}
%&Q(\alp_1h_1-a_2h_2, b_1 h_1-b_2h_2)=(A_1 a_1^2+A_2 a_1b_1+A_3 b_1^2)h_1^2\\
%& -(A_1 a_1a_2+A_2(a_2b_1+a_1b_2)+A_3b_1b_2)h_1h_2+A_1 a_2^2+A_2 a_2 b_2 +A_3 b_2^2.
%\end{split}
%\end{equation*}
%and for $i=-2,-1,0, \cdots,s-1$
$$\deg\alpha_i \geq \deg \beta_i$$ and 
$$\deg \alpha_i \geq \deg \alpha_{i-1}.$$
Thus, the highest degree term comes from the summand with $i=\ell-1$.
Putting altogether, we obtain the denominator $C_k$ as well as the degree $m= kC+D$ 
for the explicitly given $C$, $D$. 
\end{proof}
\begin{remark}
One should notice the independence of $n$ of the denominator $C_k$ of $\zeta(-k,\fb_n)$.
\textit{A priori} this is invariant in the family. It is important to control the denominator to 
%The denominator should be controlled to define the associated 
interpolate the associated $p$-adic zeta function 
from the values at negative integers(cf. \cite{Coates-S}, \cite{Deligne-Ribet} and \cite{katz_another}).
\end{remark}

 For the rest of the paper, for a number field $K$, let us denote the ring of integers by $O_K$. 

\begin{exam}Consider the family $(K_n=\FQ(\sqrt{n^2+2}), \fb_n= O_{K_n})$. 
Then 
$$
\fb_n^{-1} =O_{K_n}=[1,\omega_n]
$$ for
%$\omega_n = \sqrt {n^2+2} + n $.\\
$\omega_n=[[2n,n]]$.
%\bigskip

Then we have
\begin{align*}
\zeta(0,\fb_n)&=\tfrac{n}{12}\\
\zeta(-1,\fb_n)&=-\tfrac{19 n}{360} + \tfrac{n^3}{40}\\
\zeta(-2,\fb_n)&=\tfrac{2 n}{45} -\tfrac{ n^3}{945} - \tfrac{23 n^5}{1890}\\
\zeta(-3,\fb_n)&=-\tfrac{2159 n}{25200} + \tfrac{137 n^3}{25200} - \tfrac{59 n^5}{840} + \tfrac{3 n^7}{56}\\
\zeta(-4,\fb_n)&=
	\tfrac{68 n}{231} - \tfrac{797 n^3}{6930} + \tfrac{689 n^5}{1155} + \tfrac{134 n^7}{1155} - \tfrac{
 2878 n^9}{10395}\\
 \zeta(-5,\fb_n)&=
 	-\tfrac{11947883 n}{7567560} + \tfrac{29660563 n^3}{22702680} - \tfrac{
 26073083 n^5}{5675670} -\tfrac{7603 n^7}{2310} - \tfrac{145933 n^9}{135135} 
 	+ \tfrac{
 351719 n^{11}}{135135}
\end{align*}

\end{exam}
\begin{exam}
Let $K_n=\FQ(\sqrt{16 n^4+32 n^3+24 n^2+12n +3})$ and $\fb_n=O_{K_n}$.
Then $\fb_n^{-1} =O_{K_n}=[1,\omega_n]$  for 
%$\omega_n = \sqrt {n^2+2} + n $.\\
$\omega_n=[[8n^2+8n+2,2n+1]]$.
{\allowdisplaybreaks
\begin{align*}
\zeta(0,\fb_n)=&\tfrac{1}{12} +\tfrac{ n}{2} + \tfrac{2 n^2}3\\
\zeta(-1,\fb_n)=&-\tfrac{7}{72} - \tfrac{13 n}{20} - \tfrac{11 n^2}9 + \tfrac{n^3}{45} + \tfrac{34 n^4}{15} + \tfrac{
 104 n^5}{45} + \tfrac{32 n^6}{45}\\
\zeta(-2,\fb_n)=&\tfrac{ 503}{2520} + \tfrac{2773 n}{1260} + \tfrac{8473 n^2}{945} + \tfrac{13009 n^3}{945} - \tfrac{6898 n^4}{945} - \tfrac{360 n^5}7 - \tfrac{6208 n^6}{105} \\
& - \tfrac{3328 n^7}{315} + \tfrac{
 25472 n^8}{945} + \tfrac{18944 n^9}{945} + \tfrac{4096 n^{10}}{945}\\
\zeta(-3,\fb_n)=&-\tfrac{823}{840} - \tfrac{7762 n}{525} - \tfrac{193469 n^2}{2100} - \tfrac{309377 n^3}{1050} - \tfrac{232553 n^4}{525} + \tfrac{143188 n^5}{1575} + \tfrac{2707724 n^6}{1575} \\ 
& + 
\tfrac{5759672 n^7}{1575}  + \tfrac{7377392 n^8}{1575} + \tfrac{7421248 n^9}{1575}
  + \tfrac{
 147072 n^{10}}{35} + \tfrac{4862464 n^{11}}{1575} + \tfrac{830464 n^{12}}{525}\\
 & + \tfrac{
 249856 n^{13}}{525} + \tfrac{32768 n^{14}}{525}
 \\
\zeta(-4,\fb_n)=&\tfrac{106613}{11880} + \tfrac{262407 n}{1540} + \tfrac{1957759 n^2}{1386} + \tfrac{
 23147174 n^3}{3465} + \tfrac{203979376 n^4}{10395} + \tfrac{
 365417032 n^5}{10395} \\  & + \tfrac{257724232 n^6}{10395} - 
 \tfrac{764543312 n^7}{10395} - \tfrac{1238665888 n^8}{3465} - \tfrac{
 3134586496 n^9}{3465} - \tfrac{3314036224 n^{10}}{2079}\\
 & - \tfrac{
 4177427456 n^{11}}{2079}  
 - \tfrac{880111616 n^{12}}{495}  - \tfrac{
 2179907584 n^{13}}{2079} - \tfrac{27426816 n^{14}}{77} \\
 & - \tfrac{86638592 n^{15}}{3465} 
 + \tfrac{
 22151168 n^{16}}{693}
  + \tfrac{12058624 n^{17}}{945}  + \tfrac{16777216 n^{18}}{10395}\\
 \zeta(-5,\fb_n)=&-\tfrac{4617527}{34398} - \tfrac{52647823 n}{17199} - \tfrac{71296254191 n^2}{2270268} - \tfrac{
 22288517357 n^3}{115830} - \tfrac{2248765926611 n^4}{2837835} 
  \\
   & - \tfrac{
 202240251208 n^5}{85995} 
  - \tfrac{1639280941052 n^6}{315315} - \tfrac{
 7365379306328 n^7}{945945} - \tfrac{45593045200 n^8}{51597}\\
 & + \tfrac{
 124808351658752 n^9}{2837835} 
 + \tfrac{501230433622144 n^{10}}{2837835} + \tfrac{
 1216530615292672 n^{11}}{2837835}+ 
 \tfrac{421424974443008 n^{12}}{567567} \\
 & + \tfrac{
 2733316068964352 n^{13}}{2837835} + \tfrac{301312337145856 n^{14}}{315315} + \tfrac{
 694067022135296 n^{15}}{945945} + \tfrac{416253219504128 n^{16}}{945945} 
 \\ &+ \tfrac{
 84428067700736 n^{17}}{405405} 
 + \tfrac{223080016510976 n^{18}}{2837835} + \tfrac{
 13420435865600 n^{19}}{567567} + \tfrac{15443400065024 n^{20}}{2837835} \\
 & + \tfrac{
 481111834624 n^{21}}{567567}  + \tfrac{185488900096 n^{22}}{2837835}
\end{align*}
}

\end{exam}

\end{document}